\documentclass[11pt]{amsart}
\usepackage{}
\usepackage{mathrsfs}
\usepackage{cases}
\usepackage{amsmath}
\usepackage{amssymb}
\usepackage{amsfonts}
\usepackage{graphicx}
\newtheorem{theorem}{Theorem}[section]

\newtheorem{corollary}[theorem]{Corollary}

\newtheorem*{problem*}{Problem 1}

\newtheorem{remark}[theorem]{Remark}

\newtheorem{thm}{Theorem}

\newcommand{\sn}{\mathrm{sn}}

\begin{document}
\title[Comparison geometry for integral Bakry-\'Emery Ricci tensor]
{Comparison geometry for integral Bakry-\'Emery Ricci tensor bounds}

\author{Jia-Yong Wu}
\address{Department of Mathematics, Shanghai Maritime University, Shanghai 201306, P. R. China}
\email{jywu81@yahoo.com}

\date{\today}

\begin{abstract}
We prove mean curvature and volume comparison estimates on smooth metric measure
spaces when their integral Bakry-\'{E}mery Ricci tensor bounds, extending Wei-Wylie's
comparison results to the integral case. We also apply comparison results to get
diameter estimates, eigenvalue estimates and volume growth estimates on smooth metric
measure spaces with their normalized integral smallness for Bakry-\'{E}mery Ricci
tensor. These give generalizations of some work of Petersen-Wei, Aubry, Petersen-Sprouse,
Yau and more.
\end{abstract}
\maketitle

\section{Introduction and main results}\label{Int1}
In \cite{[PeWe]}, Petersen and Wei generalized the classical relative Bishop-Gromov
volume comparison to a situation where one has an integral bound for the Ricci tensor.
Let's briefly recall their results. Given an $n$-dimensional complete Riemannian
manifold $M$, for each $x\in M$ let $\lambda\left( x\right) $ be the smallest
eigenvalue for the Ricci tensor $\mathrm{Ric}:T_xM\to T_xM,$ and
\[
\mathrm{Ric}^H_-(x):=((n-1)H-\lambda(x))_+=\max\left\{0,(n-1)H-\lambda(x)\right\},
\]
where $H\in \mathbb{R}$, the amount of Ricci tensor below $(n-1)H$. Define
\[
\|\mathrm{Ric}^H_-\|_p(R):=\sup_{x\in M}\left(\int_{B(x,R)}(\mathrm{Ric}^H_-)^p\,
dv\right)^{\frac 1p},
\]
which measures the amount of Ricci tensor lying below $(n-1)H$, in the $L^p$
sense. Clearly, $\|\mathrm{Ric}^H_-\|_p(R)=0$ iff ${\mathrm{Ric}}\geq (n-1)H$.
Also let $r(y)=d(y,x)$ be the distance function from $x$ to $y$, and
\[
\varphi(y):=(\Delta r-m_H)_+,
\]
where $m_H$ is the mean curvature of the geodesic sphere in $M^n_H$, the
$n$-dimensional simply connected space with constant sectional curvature $H$.
The classical Laplacian comparison states that if $\mathrm{Ric}\ge(n-1)H$,
then $\Delta\,r \le m_H$. That is to say, if $\mathrm{Ric}^H_- \equiv 0$,
then $\varphi\equiv 0$. In fact this comparison result was generalized to
integral Ricci tensor lower bound.

\begin{thm}[Petersen-Wei \cite{[PeWe]}]\label{T1}
Let $M$ be an $n$-dimensional complete Riemannian manifold. For any
$p>\frac{n}{2}$, $H\in\mathbb{R}$ (assume $r\leq\frac{\pi}{2\sqrt{H}}$ when $H>0$),
\[
\|\varphi\|_{2p}(r)\leq\left[\frac{(n-1)(2p-1)}{2p-n}\,\|\mathrm{Ric}^H_-\|_p(r)\right]^{\frac 12}.
\]
Consequently, for any $0<r\leq R$ (assume $R\leq\frac{\pi}{2\sqrt{H}}$ when $H>0$),
there exists a constant $C(n,p,H,R)$ which is nondecreasing in $R$, such that
\[
\left(\frac{V(x, R)}{V_H(R)}\right)^{\frac{1}{2p}}-\left(\frac{V(x,r)}{V_H(r)}\right)^{\frac{1}{2p}}\leq C(n,p,H,R)\,\bigg(\|\mathrm{Ric}^H_-\|_p(R)\bigg)^{\frac 12},
\]
where $V(x,R)$ denotes the volume of ball $B(x,R)$ in $M$, and $V_H(R)$ denotes
the volume of ball $B(O,R)$ in the model space $M_H$, where $O\in M_H$.
\end{thm}

Petersen and Wei \cite{[PeWe],PeWe00} used these comparison estimates to extend
many classical results of pointwise Ricci tensor condition to the integral
curvature condition, such as compactness theorems, Colding's volume convergence
and Cheeger-Colding splitting theorems. Petersen and Sprouse \cite{[PeSp]}
extended Petersen-Wei's comparison results and generalized Myers' theorem to
a integral Ricci tensor bound. Aubry \cite{[Au]} used integral comparison
estimates on star-shaped domains to improve Petersen-Sprouse's diameter
estimate. He also got finite fundamental group theorem in the integral Ricci
tensor sense. For more results, see for example \cite {[Au], [Au2], [DPW], [DW],
[DWZ], [Gal], PeWe00, [Yang]}.

\vspace{.1in}

An $n$-dimensional smooth metric measure space, denoted by $(M,g,e^{-f}dv_g)$,
is a complete $n$-dimensional Riemannian manifold $(M,g)$ coupled with a weighted
volume $e^{-f}dv_g$ for some $f\in C^\infty(M)$, where $dv_g$ is the usual Riemannian
volume element on $M$. It naturally occurs as the collapsed measured Gromov-Hausdorff
limit \cite{[Lott1]}. The $f$-Laplacian $\Delta_f$ associated to $(M,g,e^{-f}dv_g)$
is given by
\[
\Delta_f:=\Delta-\nabla f\cdot\nabla,
\]
which is self-adjoint with respect to $e^{-f}dv_g$. The associated Bakry-\'Emery
Ricci tensor, introduced by Bakry and \'Emery \cite{[BE]}, is defined as
\[
\mathrm{Ric}_f:=\mathrm{Ric}+\mathrm{Hess}\,f,
\]
where $\mathrm{Hess}$ is the Hessian with respect to the metric $g$, which is a natural
generalization of the Ricci tensor. In particular, if $$\mathrm{Ric}_f=\rho\, g$$
for some $\rho\in\mathbb{R}$, then $(M,g,e^{-f}dv_g)$ is a gradient Ricci soliton.
The Ricci soliton is called shrinking, steady, or expanding, if $\rho>0$, $\rho=0$,
or $\rho<0$, respectively, which arises as the singularity model of the Ricci flow
\cite{[Ham]}. When $\mathrm{Ric}_f$ is bounded below, many geometrical and
topological results were successfully explored provided some condition on $f$
is added. For example, Wei and Wylie \cite{[WW]} proved mean curvature and volume
comparisons when $\mathrm{Ric}_f$ is bounded below and $f$ or $\nabla f$ is (lower)
bounded. And they extended many classical theorems, such as Myers' theorem,
Cheeger-Gromoll splitting theorem, to the Bakry-\'Emery Ricci tensor. They also
expected volume comparisons could be extended to the case that $\mathrm{Ric}_f$ is
bounded below in the integral sense, which partly motivates the present paper.

In this paper we not only generalize comparison estimates on manifolds with
integral bounds for the Ricci tensor to smooth metric measure spaces, but also
extend pointwise comparison estimates on smooth metric measure spaces to the
integral setting. In our situation, we consider weighted integral bounds for
the Bakry-\'{E}mery Ricci tensor instead of usual integral bounds for the
Ricci tensor. Our results indicate that Petersen-Wei's and Aubry's type
comparison estimates remain true when certain weighted integral
Bakry-\'{E}mery Ricci tensor bounds and $\nabla f$ is lower bounded (even no
assumption on $f$). We also prove an relative weighted (or $f$-)volume comparison
for annular regions under the same curvature integral condition. Some applications,
such as diameter estimates, eigenvalue estimates and volume growth estimates,
are discussed.

\vspace{.1in}

Fix $H\in\mathbb{R}$, and consider at each point $x$ of an $n$-dimensional smooth
metric measure space $(M, g, e^{-f}dv_g)$ with the smallest eigenvalue $\lambda(x)$
for the tensor $\mathrm{Ric}_f:T_xM\to T_xM$. We define
\[
{\mathrm{Ric}^H_f}_-:=[(n-1)H-\lambda(x)]_+=\max\{0,(n-1)H-\lambda(x)\},
\]
the amount of $\mathrm{Ric}_f$ lying below $(n-1)H$. To write our results simplicity,
we introduce a new weighted $L^p$ norm of function $\phi$ on $(M, g, e^{-f}dv_g)$:
\[
{\|\phi\|_p}_{\,f,a}(r):=\sup_{x\in M}
\left(\int_{B(x,r)}|\phi|^p\cdot\mathcal{A}_f e^{-at}\,dtd\theta_{n-1}\right)^{\frac 1p},
\]
where $\partial_rf\geq-a$ for some constant $a\geq 0$, along a minimal geodesic segment
from $x\in M$. Here $\mathcal{A}_f(t,\theta)$ is the volume element of weighted form
$e^{-f}dv_g=\mathcal{A}_f(t,\theta)dt\wedge d\theta_{n-1}$ in polar coordinate, and
$d\theta_{n-1}$ is the volume element on unit sphere $S^{n-1}$. Sometimes it is
convenient to work with the normalized curvature quantity
\[
\bar{k}(p,H,a,r):=\sup_{x\in M}
\left(\frac{1}{V_f(x,r)}\cdot\int_{B(x,r)}({{\mathrm{Ric}^H_f}_-})^p
\mathcal{A}_f e^{-at}dtd\theta_{n-1}\right)^{\frac 1p},
\]
where $V_f(x,r):=\int_{B(x,r)}e^{-f}dv$. Obviously,
${\|{\mathrm{Ric}^H_f}_-\|_p}_{\,f,a}(r)=0$ (or $\bar{k}(p,H,a,r)=0$)
iff ${\mathrm{Ric}_f}\geq (n-1)H$. When $f=0$ (and $a=0$), all above notations
recover the usual integral quantities on manifolds.

Motivated by Wei-Wylie's mean curvature comparison \cite{[WW]}, we need
to consider the error form
\[
\varphi:=(m_f-m_H-a)_+,
\]
where $m_f=m-\partial_rf$ and $m$ is the mean curvature of the geodesic sphere
in the outer normal direction; and where $m_H$ is the mean curvature of the
geodesic sphere in the model space $M^n_H$. In \cite{[WW]}, Wei and Wylie
showed that if ${\mathrm{Ric}^H_f}_-=0$ and $\partial_r f\geq-a$  $(a\ge0)$,
then $\varphi=0$. We prove that,

\begin{theorem}[Mean curvature comparison estimate I]\label{Mainthm}
Let $(M,g,e^{-f}dv)$ be an $n$-dimensional smooth metric measure space. Assume that
\[
\partial_rf\geq-a
\]
for some constant $a\geq 0$, along a minimal geodesic segment from $x\in M$.
For any $p>n/2$, $H\in\mathbb{R}$ (assume $r\leq\frac{\pi}{2\sqrt{H}}$ when $H>0$),
\begin{equation}\label{wintegine}
{\|\varphi\|_{2p}}_{\,f,a}(r)
\leq\left[\frac{(n-1)(2p-1)}{(2p-n)}\,{\|{\mathrm{Ric}^H_f}_-\|_p}_{\,f,a}(r)\right]^{\frac 12}
\end{equation}
and
\begin{equation}\label{wintegine2}
\varphi^{2p-1}\mathcal{A}_f\, e^{-ar}
\leq(2p-1)^p\left(\frac{n-1}{2p-n}\right)^{p-1}\cdot\int^r_0({\mathrm{Ric}^H_f}_-)^p \mathcal{A}_fe^{-at}dt
\end{equation}
along that minimal geodesic segment from $x$.

\vspace{.1in}

Moreover, if $H>0$ and $\frac{\pi}{2\sqrt{H}}<r<\frac{\pi}{\sqrt{H}}$, then we have
\begin{equation}\label{winteginesin}
{\Big\|\sin^{\frac{4p-n-1}{2p}}(\sqrt{H}t)\cdot\varphi\Big\|_{2p}}_{\,f,a}(r)
\leq\left[\frac{(n-1)(2p-1)}{(2p-n)}\,{\|{\mathrm{Ric}^H_f}_-\|_p}_{\,f,a}(r)\right]^{\frac 12}
\end{equation}
and
\begin{equation}\label{winteginesin2}
\sin^{4p-n-1}(\sqrt{H}r)\varphi^{2p-1}\mathcal{A}_f e^{-ar}
\leq(2p{-}1)^p\left(\frac{n{-}1}{2p{-}n}\right)^{p-1}\cdot\int^r_0({\mathrm{Ric}^H_f}_-)^p \mathcal{A}_fe^{-at}dt
\end{equation}
along that minimal geodesic segment from $x$.
\end{theorem}

\begin{remark}
\item[(1)]
When $f$ is constant (and $a=0$), inequality \eqref{wintegine} recovers the
Petersen-Wei's result \cite{[PeWe]}; inequalities \eqref{wintegine2}
and \eqref{winteginesin2} recover the Aubry's results \cite{[Au]}. In particular,
when $|\nabla f|\le a$ for some constant $a\ge 0$ and the diameter of $M$ is bounded,
then $f$ is bounded, and the new weighted norm is equivalent to the usual norm.

\item[(2)]
When ${\mathrm{Ric}^H_f}_-\equiv 0$ (i.e. $\mathrm{Ric}_f\geq (n-1)H$), we
have $\varphi\equiv0$ and hence get the Wei-Wylie's comparison result \cite{[WW]}.
\end{remark}

As in the integral volume comparison for manifolds \cite{[PeWe]}, we can apply
Theorem \ref{Mainthm} to prove weighted volume comparisons in the integral sense.
Let $V_f(x,R):=\int_{B(x,R)}e^{-f}dv$ be the weighted volume of ball $B(x,R)$ in
$(M,g,e^{-f}dv)$. $V^a_H(R)$ denotes the $h$-volume of the ball
$B(O,R)$ in the weighted model space $M^n_{H,a}:=(M^n_H, g_H,e^{-h}dv_{g_H})$,
where $O\in M^n_H$ and $h(x):=-a\cdot d(O,x)$. That is,
\[
V^a_H(R):=\int^R_0\int_{S^{n-1}} e^{at}\mathcal{A}_H(t,\theta)\,d\theta_{n-1}dt=\int^R_0e^{at}A_H(t)dt,
\]
where $\mathcal{A}_H$ denotes the volume element in model space $M^n_H$,
and $A_H$ denotes the volume of the geodesic sphere in $M^n_H$. For more
detailed description about the related notations, see Section \ref{sec3}.
\begin{theorem}[Relative volume comparison estimate I]\label{volcomp}
Let $(M,g,e^{-f}dv)$ be an $n$-dimensional smooth metric measure space.
Assume that
\[
\partial_rf\geq-a
\]
for some constant $a\geq 0$, along all minimal geodesic segments from $x\in M$. Let
$H\in\mathbb{R}$ and $p>n/2$. For $0<r\leq R$ (assume $R\leq\frac{\pi}{2\sqrt{H}}$
when $H>0$),
\begin{equation*}\label{volcompineq}
\left(\frac{V_f(x,R)}{V^a_H(R)}\right)^{\frac{1}{2p-1}}-\left(\frac{V_f(x,r)}{V^a_H(r)}\right)^{\frac{1}{2p-1}}
\leq C(n,p,H,a,R)\left({\|{\mathrm{Ric}^H_f}_-\|^p_p}_{\,f,a}(R)\right)^{\frac{1}{2p-1}}.
\end{equation*}
Furthermore, when $r=0$, we have an absolute volume comparison estimate:
\[
V_f(x,R) \leq \bigg[e^{-\frac{f(x)}{2p-1}}+C(n,p,H,a,R)
\left({\|{\mathrm{Ric}^H_f}_-\|^p_p}_{\,f,a}(R)\right)^{\frac{1}{2p-1}}\bigg]^{2p-1}V^a_H(R).
\]
Here,
$
C(n,p,H,a,R):=\left(\frac{n-1}{(2p-1)(2p-n)}\right)^{\frac{p-1}{2p-1}}
\int^R_0A_H(t)\left(\frac{t\,e^{at}}{V^a_H(t)}\right)^{\frac{2p}{2p-1}}dt.
$
\end{theorem}

\begin{remark}
\item[(1)]
The theorem implies a useful volume doubling property, see Corollary \ref{corvde}
below. When $f$ is constant (or furthermore $f=0$) and $a=0$, the theorem recovers
the Petersen-Wei's result \cite{[PeWe]}.
\item[(2)]
When ${\mathrm{Ric}^H_f}_-\equiv 0$, i.e. $\mathrm{Ric}_f\geq (n-1)H$, we have
the Wei-Wylie's volume comparison result (see (4.10) in \cite{[WW]}).
\item[(3)]
Integrating along the direction lies in a star-shaped domain at $x$, we can obtain
the same volume comparison estimate for the star-shaped domain at $x$, where
${\mathrm{Ric}^H_f}_-$ only needs to integrate on the same star-shaped set.
\end{remark}

We can generalize Theorem \ref{volcomp} and get an relative weighted
volume comparison for two annuluses in the integral sense, which is completely new
even in the manifold case. Let $V_f(x,r,R)$ be the $f$-volume of the annulus
$B(x,R)\backslash B(x,r)\subseteq M^n$ for $r\le R$, and $V^a_H(r,R)$ be the
$h$-volume of the annulus $B(O,R)\backslash B(O,r)\subseteq M^n_{H,a}$.

\begin{theorem}[Relative volume comparison for annulus]\label{vocoannu}
Let $(M,g,e^{-f}dv)$ be an $n$-dimensional smooth metric measure space.
Assume that
\[
\partial_rf\geq-a
\]
for some constant $a\geq 0$, along all minimal geodesic segments from $x\in M$. Let
$H\in\mathbb{R}$ and $p>n/2$. For  $0\leq r_1\leq r_2\leq R_1\leq R_2$ (assume
$R_2\leq\frac{\pi}{2\sqrt{H}}$ when $H>0$),
\[
\left(\frac{V_f(x,r_2,R_2)}{V^a_H(r_2,R_2)}\right)^{\frac{1}{2p-1}}-\left(\frac{V_f(x,r_1,R_1)}{V^a_H(r_1,R_1)}\right)^{\frac{1}{2p-1}}\leq \mathcal{C}\cdot\left({\|{\mathrm{Ric}^H_f}_-\|^p_p}_{\,f,a}(R_2)\right)^{\frac{1}{2p-1}},
\]
where $\mathcal{C}$ is given by
\begin{equation*}
\begin{aligned}
\mathcal{C}&=C(n,p,H,a,r_1,r_2,R_1,R_2):=\left(\frac{n-1}{(2p-1)(2p-n)}\right)^{\frac{p-1}{2p-1}}\\
&\quad\times\left[\int^{r_2}_{r_1}A_H(R_1)\left(\frac{R_1\,e^{aR_1}}{V^a_H(t,R_1)}\right)^{\frac{2p}{2p-1}}dt
+\int^{R_2}_{R_1}A_H(t)\left(\frac{t\,e^{at}}{V^a_H(r_2,t)}\right)^{\frac{2p}{2p-1}}dt\right].
\end{aligned}
\end{equation*}
\end{theorem}

\vspace{.1in}

Besides, we are able to prove a general mean curvature comparison estimate, requiring
no assumptions on $f$. Consequently, we get relative volume comparison estimates
when $f$ is bounded. See these results in Section \ref{sec4}.

\vspace{.1in}

The integral comparison estimates have many applications. We start to highlight
two extensions of Petersen-Sprouse's results \cite{[PeSp]} to the weighted case
that $\nabla f$ is lower bounded. One is the global diameter estimate:
\begin{theorem}\label{diam}
Let $(M,g,e^{-f}dv)$ be an $n$-dimensional smooth metric measure space.
Assume that $$\partial_rf\geq-a$$ for some constant $a\geq 0$, along all minimal
geodesic segments from any $x\in M$. Given $p>n/2$, $H>0$ and $R>0$, there exist
$D=D(n,H,a)$ and $\epsilon=\epsilon(n,p,a,H,R)$ such that if
$\bar{k}(p,H,a,R)<\epsilon$, then $diam_M\le D$.
\end{theorem}

This theorem shows that a small fluctuation of super gradient shrinking Ricci
soliton (i.e. $\mathrm{Ric}_f\ge(n-1)H g$ for some constant $H>0$) must be compact
provided that the derivative of $f$ has a lower bound. Examples 2.1 and 2.2 in
\cite {[WW]} indicate that the assumption of $f$ is necessary. Petersen and
Sprouse \cite{[PeSp]} have proved the case when $f$ is constant. For other
Myers' type theorems on smooth metric measure spaces, see \cite{[WW], [Lim], [Tad]}.

\vspace{.1in}

The other is a generalization of Cheng's eigenvalue upper bounds \cite{[Cheng]}.
For any point $x_0\in(M,g,e^{-f}dv)$ and $R>0$, let $\lambda^D_1(B(x_0,R))$
denote the first eigenvalue of the $f$-Laplacian $\Delta_f$ with the Dirichlet
condition in $B(x_0,R)$. Let $\lambda^D_1(n,H,a,R)$ denote the first eigenvalue
of the $h$-Laplacian $\Delta_h$, where $h(x):=-a\cdot d(\bar{x}_0,x)$, with the
Dirichlet condition in a metric ball $B(\bar{x}_0,R)\subseteq M^n_{H,a}$,
where $R\leq \frac{\pi}{2\sqrt{H}}$ when $H>0$. Then, we have a weighted version of
Petersen-Sprouse's result \cite{[PeSp]}.
\begin{theorem}\label{eigen}
Let $(M,g,e^{-f}dv)$ be an $n$-dimensional smooth metric measure space.
Assume that $$\partial_rf\geq-a$$ for some constant $a\geq 0$, along all minimal
geodesic segments from $x_0\in M$. Given $p>n/2$, $H\in\mathbb{R}$ and $R>0$ (assume
$R\leq \frac{\pi}{2\sqrt{H}}$ when $H>0$), for every $\delta>0$, there exists an
$\epsilon=\epsilon(n,p,H,a,R)$ such that if $\bar{k}(p,H,a,R)\leq\epsilon$, then
\[
\lambda^D_1(B(x_0,R))\leq (1+\delta)\,\lambda^D_1(n,H,a,R).
\]
\end{theorem}

\vspace{.1in}

Finally we apply Theorem \ref{vocoannu} to get a weighted volume growth estimate,
generalizing Yau's volume growth estimate \cite{[Yau]} and Wei-Wylie's result \cite{[WW]}.
\begin{theorem}\label{volumegrow}
Let $(M,g,e^{-f}dv)$ be an $n$-dimensional smooth metric measure space.
Assume that
\[
\partial_rf\geq 0
\]
along all minimal geodesic segments from any $x\in M$. Given any $p>n/2$ and $R\ge 2$,
there is an $\epsilon=\epsilon(n,p,R)$ such that if $\bar{k}(p,0,0,R+1)<\epsilon$
(here $H=0$ and $a=0$), then for any point $x_0\in M$, we have
\[
V_f(x_0,R)\geq C\, R
\]
for some positive constant $C=C(n,p,V_f(x_0,1))$ depending only on $n$, $p$
and $V_f(x_0,1)$.
\end{theorem}

Constant function $f$ satisfies $\partial_rf\geq 0$ and hence the theorem
naturally holds for the ordinary Riemannian manifolds. From Theorem 5.3
in \cite {[WW]}, we know that convex function $f$ with the unbounded set of
its critical points, also satisfies $\partial_rf\geq 0$. Examples 2.1 and 2.2
in \cite {[WW]} indicate that the hypothesis on $f$ in the theorem is necessary.

\vspace{.1in}

The rest of this paper is organized as follows. In Section \ref{sec2},
we will prove Theorem \ref{Mainthm}. In Section \ref{sec3}, we will
apply Theorem \ref{Mainthm} to prove Theorem \ref{volcomp} and further
get a volume doubling property when the integral Bakry-\'{E}mery Ricci
tensor bounds and $\nabla f$ is lower bounded. We also prove relative
volume comparison estimates for annuluses when the integral of
Bakry-\'{E}mery Ricci tensor bounds. In Section \ref{sec4}, we will
discuss a general mean curvature comparison estimates and relative volume
comparison estimates for their integral bounds of Bakry-\'{E}mery Ricci
tensor. In Section \ref{sec5}, we will give some applications of new integral
comparison estimates. Precisely, we will apply Theorems \ref{Mainthm} and
\ref{volcomp} to prove Theorems \ref{diam} and \ref{eigen}. Meanwhile we will
apply Theorem \ref{vocoannu} to prove Theorem \ref{volumegrow}. In Appendix,
we give mean curvature and volume comparison estimates on smooth metric measure
spaces when only certain integral of $m$-Bakry-\'{E}mery Ricci tensor bounds.

\vspace{.1in}

From the work of \cite{[Au],[Au2],[HuXu]} we expect Aubry's type diameter estimate,
finiteness fundamental group theorem, first Betti number estimate and Gromov's
bounds on the volume entropy in the integral sense can be generalized to smooth
metric measure spaces. These will be treated in separate paper.

\vspace{.1in}

\textbf{Acknowledgement}.
The author sincerely thanks Professor Guofang Wei for her useful note,
valuable advices, stimulating discussions and bringing my attention to the paper
\cite{[HuXu]}. The author also thanks Peng Wu for many helpful discussions.
This work was partially supported by the Natural Science Foundation of Shanghai (No. 17ZR1412800)
and the National Natural Science Foundation of China (No. 11671141).

\section{Mean curvature comparison estimate I}\label{sec2}

In this section, we mainly prove Theorem \ref{Mainthm}, a weighted mean curvature
comparison estimate on smooth metric measure spaces $(M,g,e^{-f}dv)$ when certain
integral Bakry-\'{E}mery Ricci tensor bounds and $\nabla f$ is lower bounded.
The proof first modifies the Bochner formula of Bakry-\'Emery Ricci tensor to
acquire the ODE along geodesics and then integrates the ODE inequality, similar
to the arguments of Petersen and Wei \cite{[PeWe]}, and Aubry \cite{[Au]}.

\begin{proof}[Proof of Theorem \ref{Mainthm}]
Recall the Bochner formula
\[
\frac 12\Delta|\nabla u|^2=|\mathrm{Hess}\,u|^2+\langle\nabla u,\nabla(\Delta u\rangle)
+\mathrm{Ric}(\nabla u, \nabla u)
\]
for any function $u\in C^\infty(M)$. Letting $u=r(y)$, where $r(y)=d(y,x)$ is the distance
function, then we have
\[
0=|\mathrm{Hess}\,r|^2+\frac{\partial}{\partial r}(\Delta r)+\mathrm{Ric}(\nabla r,\nabla r).
\]
Note that $\mathrm{Hess}\,r$ is the second fundamental from of the geodesic
sphere and $\Delta r=m$, the mean curvature of the geodesic sphere. By the
Schwarz inequality, we have the Riccati inequality
\[
m'\leq-\frac{m^2}{n-1}-\mathrm{Ric}(\partial r, \partial r).
\]
This inequality becomes equality if and only if the radial sectional curvatures are
constant. So the mean curvature of the $n$-dimensional model space $m_H$ satisfies
\[
m'_H=-\frac{m_H^2}{n-1}-(n-1)H.
\]
Since $m_f:=m-\partial_r f$, i.e. $m_f=\Delta_f\, r$, then
$m_f'=m'-\partial_r\partial_r f$,
and we have
\[
 m_f'\leq-\frac{m^2}{n-1}-\mathrm{Ric}_f(\partial r, \partial r).
\]
Hence,
\begin{equation*}
\begin{aligned}
(&m_f-m_H-a)'=m'_f-m'_H\\
&\leq-\frac{m^2-m_H^2}{n-1}+(n-1)H-\mathrm{Ric}_f\\
&=-\frac{(m_f+\partial_rf)^2-m_H^2}{n-1}+(n-1)H-\mathrm{Ric}_f\\
&=-\frac{1}{n-1}\Big[(m_f-m_H+\partial_rf)(m_f+m_H+\partial_rf)\Big]+(n-1)H-\mathrm{Ric}_f\\
&=-\frac{1}{n-1}\Big[(m_f-m_H-a+a+\partial_rf)(m_f-m_H-a+2m_H+a+\partial_rf)\Big]\\
&\quad+(n-1)H-\mathrm{Ric}_f.
\end{aligned}
\end{equation*}
We recall that $\varphi:=(m_f-m_H-a)_+$. Notice that on the interval where $m_f\leq m_H+a$,
we have $\varphi=0$; on the interval where $m_f>m_H+a$, we have $m_f-m_H-a=\varphi$.
Moreover, by our assumption of the theorem, we know
\[
(n-1)H-\mathrm{Ric}_f\leq{\mathrm{Ric}^H_f}_-\,.
\]
Therefore, in any case, we have
\[
\varphi'+\frac{1}{n-1}\Big[(\varphi+a+\partial_rf)(\varphi+2m_H+a+\partial_rf)\Big]\leq {\mathrm{Ric}^H_f}_-.
\]
Since $a+\partial_r f\geq 0$, the above inequality implies
\[
\varphi'+\frac{\varphi^2}{n-1}+\frac{2m_H\varphi}{n-1}\leq {\mathrm{Ric}^H_f}_-.
\]
Multiplying this inequality by $(2p-1)\varphi^{2p-2}\cdot\mathcal{A}_f$, we have
\begin{equation*}
\begin{aligned}
(2p-1)\varphi^{2p-2}\varphi'\mathcal{A}_f+\frac{2p-1}{n-1}\varphi^{2p}&\mathcal{A}_f
+\frac{4p-2}{n-1}\varphi^{2p-1}m_H\mathcal{A}_f\\
&\quad\leq (2p-1){\mathrm{Ric}^H_f}_-\cdot\varphi^{2p-2}\mathcal{A}_f.
\end{aligned}
\end{equation*}
Using
\begin{equation*}
\begin{aligned}
(\varphi^{2p-1}\mathcal{A}_f)' &=(2p-1)\varphi^{2p-2}\varphi'\cdot\mathcal{A}_f+\varphi^{2p-1}\cdot\mathcal{A}'_f\\
&=(2p-1)\varphi^{2p-2}\varphi'\cdot\mathcal{A}_f+\varphi^{2p-1}\cdot m_f\mathcal{A}_f,
\end{aligned}
\end{equation*}
the above integral inequality can be rewritten as
\begin{equation*}
\begin{aligned}
(\varphi^{2p-1}\mathcal{A}_f)'&-\varphi^{2p-1}(m_f-m_H-a+m_H+a)\mathcal{A}_f+\frac{2p-1}{n-1}\varphi^{2p}\mathcal{A}_f\\
&+\frac{4p-2}{n-1}\varphi^{2p-1} m_H\cdot\mathcal{A}_f
\leq (2p-1){\mathrm{Ric}^H_f}_-\cdot\varphi^{2p-2} \mathcal{A}_f.
\end{aligned}
\end{equation*}
Rearrange some terms of the above inequality by $\varphi:=(m_f-m_H-a)_+$ to get
\begin{equation*}
\begin{aligned}
(\varphi^{2p-1}\mathcal{A}_f)'+\left(\frac{2p-1}{n-1}-1\right)&\varphi^{2p}\mathcal{A}_f
+\left(\frac{4p-2}{n-1}-1\right)\varphi^{2p-1}\cdot m_H\mathcal{A}_f\\
&-a\varphi^{2p-1}\mathcal{A}_f
\leq (2p-1){\mathrm{Ric}^H_f}_-\cdot\varphi^{2p-2}\mathcal{A}_f.
\end{aligned}
\end{equation*}
Notice that the term $-a\varphi^{2p-1}\mathcal{A}_f$ of the above inequality is negative.
To deal with this bad term, we multiply the inequality by the integrating factor $e^{-ar}$,
and get that
\begin{equation}
\begin{aligned}\label{imp-inequ}
(\varphi^{2p-1}\mathcal{A}_f e^{-ar})'+\frac{2p-n}{n-1}\varphi^{2p}\mathcal{A}_fe^{-ar}
&+\frac{4p-n-1}{n-1}\varphi^{2p-1} m_H \mathcal{A}_fe^{-ar}\\
&\leq (2p-1){\mathrm{Ric}^H_f}_-\cdot\varphi^{2p-2} \mathcal{A}_fe^{-ar}.
\end{aligned}
\end{equation}
Since $p>n/2$ and the assumption $r\leq\frac{\pi}{2\sqrt{H}}$, we have $m_H\geq 0$ and
\[
\frac{4p-n-1}{n-1}\varphi^{2p-1}m_H\mathcal{A}_fe^{-ar}\geq0.
\]
Then we drop this term and have that
\[
(\varphi^{2p-1} \mathcal{A}_f e^{-ar})'+\frac{2p-n}{n-1}\varphi^{2p}\mathcal{A}_fe^{-ar}
\leq (2p-1){\mathrm{Ric}^H_f}_-\cdot\varphi^{2p-2} \mathcal{A}_fe^{-ar}.
\]
We integrate the above inequality from $0$ to $r$. Since
\[
\varphi(0)=(m-m_H-\partial_r f-a)_+\big|_{r=0}=0,
\]
which comes from the theorem assumption: $a+\partial_r f\geq 0$, then
\begin{equation*}
\begin{aligned}
\varphi^{2p-1} \mathcal{A}_f e^{-ar}+\frac{2p-n}{n-1}\int^r_0&\varphi^{2p}\mathcal{A}_fe^{-at}dt\\
&\leq (2p-1)\int^r_0{\mathrm{Ric}^H_f}_-\cdot\varphi^{2p-2} \mathcal{A}_fe^{-at}dt.
\end{aligned}
\end{equation*}
This implies
\begin{equation}\label{keyinequ2}
\varphi^{2p-1} \mathcal{A}_f e^{-ar}\leq (2p-1)\int^r_0{\mathrm{Ric}^H_f}_-\cdot\varphi^{2p-2} \mathcal{A}_fe^{-at}dt.
\end{equation}
and
\begin{equation}\label{keyinequ3}
\frac{2p-n}{n-1}\int^r_0\varphi^{2p} \mathcal{A}_fe^{-at}dt
\leq (2p-1)\int^r_0{\mathrm{Ric}^H_f}_-\cdot\varphi^{2p-2}\mathcal{A}_fe^{-at}dt.
\end{equation}
By Holder inequality, we also have
\begin{equation}
\begin{aligned}\label{Holder}
\int^r_0&{\mathrm{Ric}^H_f}_-\cdot\varphi^{2p-2} \mathcal{A}_fe^{-at}dt\\
&\leq\bigg[\int^r_0\varphi^{2p} \mathcal{A}_fe^{-at}dt\bigg]^{1-\frac 1p}
\cdot\bigg[\int^r_0({\mathrm{Ric}^H_f}_-)^p \mathcal{A}_fe^{-at}dt\bigg]^{\frac 1p}.
\end{aligned}
\end{equation}
Combining \eqref{Holder} and \eqref{keyinequ3}, we immediately get \eqref{wintegine}. Then
applying \eqref{wintegine} and \eqref{Holder} to \eqref{keyinequ2} yields \eqref{wintegine2}.

\vspace{.1in}

If $H>0$ and $\frac{\pi}{2\sqrt{H}}<r<\frac{\pi}{\sqrt{H}}$, then $m_H<0$ in
\eqref{imp-inequ}. It means that we can not throw away the third term of
\eqref{imp-inequ} as before. To deal with this obstacle, multiplying by
the integrating factor $\sin^{4p-n-1}(\sqrt{H}r)$ in \eqref{imp-inequ}
and integrating from $0$ to $r$, we get
\begin{equation}
\begin{aligned}\label{ineeet}
\sin^{4p-n-1}(\sqrt{H}r)&\,\varphi^{2p-1}\mathcal{A}_f e^{-ar}
+\frac{2p-n}{n-1}\int^r_0\varphi^{2p}\sin^{4p-n-1}(\sqrt{H}t)\mathcal{A}_fe^{-at}dt\\
&\leq(2p-1)\int^r_0{\mathrm{Ric}^H_f}_-\cdot\varphi^{2p-2}\sin^{4p-n-1}(\sqrt{H}t)\mathcal{A}_fe^{-at}dt.
\end{aligned}
\end{equation}
Similar to the above discussion, using the Holder inequality, we have
\begin{equation}
\begin{aligned}\label{Holderex}
&\int^r_0{\mathrm{Ric}^H_f}_-\cdot\varphi^{2p-2}\sin^{4p-n-1}(\sqrt{H}t)\mathcal{A}_fe^{-at}dt\\
&\leq\bigg[\int^r_0\varphi^{2p}\sin^{4p-n-1}(\sqrt{H}t)\mathcal{A}_fe^{-at}dt\bigg]^{1-\frac 1p}
\bigg[\int^r_0\sin^{4p-n-1}(\sqrt{H}t)({\mathrm{Ric}^H_f}_-)^p \mathcal{A}_fe^{-at}dt\bigg]^{\frac 1p}.
\end{aligned}
\end{equation}
Notice that two terms in the left hand side of \eqref{ineeet} are both positive.
Then substituting  \eqref{Holderex} into \eqref{ineeet}, we get
\begin{equation}\label{winteginesingen}
{\Big\|\sin^{\frac{4p-n-1}{2p}}(\sqrt{H}t)\cdot\varphi\Big\|_{2p}}_{\,f,a}(r)
\leq\left[\frac{(n{-}1)(2p{-}1)}{(2p{-}n)}{\Big\|\sin^{\frac{4p-n-1}{p}}(\sqrt{H}t)
\cdot{\mathrm{Ric}^H_f}_-\Big\|_p}_{\,f,a}(r)\right]^{\frac 12},
\end{equation}
which implies \eqref{winteginesin}.
Then putting \eqref{winteginesingen} and \eqref{Holderex} to \eqref{ineeet} immediately
proves \eqref{winteginesin2} by only using an easy fact:
$\sin^{\frac{4p-n-1}{p}}(\sqrt{H}t)\leq 1$.
\end{proof}


\section{Volume comparison estimate I}\label{sec3}
In Section \ref{sec2}, we have proved a weighted mean curvature comparison estimate when
certain weighted integral of Bakry-\'{E}mery Ricci tensor bounds and $\nabla f$ has a
lower bound, and one naturally hopes a corresponding volume comparison estimate under
the same curvature assumptions. In this section, we will give these desired volume
comparison estimates.

For an $n$-dimensional smooth metric measure space $(M^n,g,e^{-f}dv_g)$, let
$\mathcal{A}_f(t,\theta)$ denote the volume element of the weighted volume form
$e^{-f}dv_g=\mathcal{A}_f(t,\theta)dt\wedge d\theta_{n-1}$ in polar coordinate.
That is, $$\mathcal{A}_f(t,\theta)=e^{-f}\mathcal{A}(t,\theta),$$ where
$\mathcal{A}(t,\theta)$ is the standard volume element of the metric $g$. We also
let
\[
A_f(x,r)=\int_{S^{n-1}}\mathcal{A}_f(r,\theta)d\theta_{n-1},
\]
which denotes the weighted volume of the geodesic sphere $S(x,r)=\{y\in M|\,d(x,y)=r\}$, and let
$A_H(r)$ be the volume of the geodesic sphere in the model space $M^n_H$. We modify $M^n_H$ to
the weighted model space
\[
M^n_{H,a}:=(M^n_H, g_H,e^{-h}dv_{g_H}, O),
\]
where $(M^n_H, g_H)$ is the $n$-dimensional simply connected space with constant
sectional curvature $H$, $O\in M^n_H$, and $h(x)=-a\cdot d(x,O)$.
Let $\mathcal{A}^a_H$ be the $h$-volume element in $M^n_{H,a}$. Then
\[
\mathcal{A}^a_H(r)=e^{ar}\mathcal{A}_H(r),
\]
where $\mathcal{A}_H$ is the Riemannian volume element in $M^n_H$. We also have that
\[
A_H(r)=\int_{S^{n-1}}\mathcal{A}_H(r,\theta)d\theta_{n-1};
\]
the corresponding weighted volume of the geodesic sphere in the weighted
model space $M^n_{H,a}$ is defined by
\[
A^a_H(r)=\int_{S^{n-1}}\mathcal{A}^a_H(r,\theta)d\theta_{n-1}.
\]
Hence,
\[
A^a_H(r)=e^{ar}A_H(r).
\]

Moreover, the weighted (or $f$-)volume of the ball
$B(x,r)=\{y\in M| d(x,y)\le r\}$ is defined by
\[
V_f(x,r)=\int^r_0A_f(x,t)dt.
\]
We also let $V^a_H(r)$ be the $h$-volume of the ball $B(O,r)\subset M^n_H$:
\[
V^a_H(r)=\int^r_0A^a_H(t)dt.
\]
Clearly, we have
\[
V_H(r)\leq V^a_H(r)\leq e^{ar}V_H(r).
\]

\

Now we prove a comparison estimate for the area of geodesic spheres using
the pointwise mean curvature estimate in Section \ref{sec2}.
\begin{theorem}\label{areacomp}
Let $(M,g,e^{-f}dv)$ be an $n$-dimensional smooth metric measure space.
Assume that
\[
\partial_rf\geq-a
\]
for some constant $a\geq 0$, along all minimal geodesic segments from $x\in M$.
Let $H\in \mathbb{R}$ and $p>n/2$ be given, and when $H>0$ assume that $R\leq \frac{\pi}{2\sqrt{H}}$.
For $0<r\leq R$, we have
\begin{equation}\label{areacompineq}
\left(\frac{A_f(x,R)}{A^a_H(R)}\right)^{\frac{1}{2p-1}}-\left(\frac{A_f(x,r)}{A^a_H(r)}\right)^{\frac{1}{2p-1}}
\leq C(n,p,H,R)\,\left({\|{\mathrm{Ric}^H_f}_-\|_p}_{\,f,a}(R)\right)^{\frac{p}{2p-1}},
\end{equation}
where $C(n,p,H,R):=\left(\frac{n-1}{(2p-1)(2p-n)}\right)^{\frac{p-1}{2p-1}}
\cdot\int^R_0A_H(t)^{-\frac{1}{2p-1}}dt$.

Moreover, if $H>0$ and $\frac{\pi}{2\sqrt{H}}<r\leq R<\frac{\pi}{\sqrt{H}}$, then we have
\begin{equation}
\begin{aligned}\label{areacompineqv2}
&\left(\frac{A_f(x,R)}{A^a_H(R)}\right)^{\frac{1}{2p-1}}-\left(\frac{A_f(x,r)}{A^a_H(r)}\right)^{\frac{1}{2p-1}}\\
&\leq\left(\frac{n{-}1}{(2p{-}1)(2p{-}n)}\right)^{\frac{p-1}{2p-1}}
\left({\|{\mathrm{Ric}^H_f}_-\|_p}_{\,f,a}(R)\right)^{\frac{p}{2p-1}}
\int^R_r\frac{(\sqrt{H})^{\frac{n-1}{2p-1}}}{\sin^2(\sqrt{H}t)} dt.
\end{aligned}
\end{equation}
\end{theorem}

\begin{remark}
When ${\mathrm{Ric}^H_f}_-\equiv 0$, that is, $\mathrm{Ric}_f\geq (n-1)H$, we exactly
get Wei-Wylie's comparison result for the area of geodesic spheres (see (4.8) in \cite{[WW]}).
\end{remark}

\begin{proof}[Proof of Theorem \ref{areacomp}]
We apply
\[
\mathcal{A}'_f=m_f \mathcal{A}_f\quad \mathrm{and}\quad
{\mathcal{A}^a_H}'=(m_H+a)\mathcal{A}^a_H
\]
to compute that
\[
\frac{d}{dt}\left(\frac{\mathcal{A}_f(t,\theta)}{\mathcal{A}^a_H(t)}\right)
=(m_f-m_H-a)\frac{\mathcal{A}_f(t,\theta)}{\mathcal{A}^a_H(t)}.
\]
Hence,
\begin{equation*}
\begin{aligned}
\frac{d}{dt}\left(\frac{A_f(x,t)}{A^a_H(t)}\right)
&=\frac{1}{Vol(S^{n-1})}\int_{S^{n-1}}\frac{d}{dt}\left(\frac{\mathcal{A}_f(t,\theta)}{\mathcal{A}^a_H(t)}\right)d\theta_{n-1}\\
&\leq\frac{1}{A^a_H(t)}\int_{S^{n-1}}\varphi\cdot\mathcal{A}_f(t,\theta)d\theta_{n-1}.
\end{aligned}
\end{equation*}
Using Holder's inequality and \eqref{wintegine2}, we have
\begin{equation*}
\begin{aligned}
\int_{S^{n-1}}\varphi\cdot&\mathcal{A}_f(t,\theta)d\theta_{n-1}\\
&\leq\left(\int_{S^{n-1}}\varphi^{2p-1} \mathcal{A}_f(x,t)d\theta_{n-1}\right)^{\frac{1}{2p-1}}
\cdot A_f(x,t)^{1-\frac{1}{2p-1}}\\
&\leq C(n,p)\,e^{\frac{at}{2p-1}} \left({\|{\mathrm{Ric}^H_f}_-\|_p}_{\,f,a}(t)\right)^{\frac{p}{2p-1}}
\cdot A_f(x,t)^{1-\frac{1}{2p-1}},
\end{aligned}
\end{equation*}
where $C(n,p)=\left[(2p-1)^p\left(\frac{n-1}{2p-n}\right)^{p-1}\right]^{\frac{1}{2p-1}}$. Hence, we have
\begin{equation}
\begin{aligned}\label{keyinequ}
\frac{d}{dt}\left(\frac{A_f(x,t)}{A^a_H(t)}\right)\leq C(n,p)&\left(\frac{A_f(x,t)}{A^a_H(t)}\right)^{1-\frac{1}{2p-1}}\\
&\times\left({\|{\mathrm{Ric}^H_f}_-\|_p}_{\,f,a}(t)\right)^{\frac{p}{2p-1}}
\cdot\left(\frac{e^{at}}{A^a_H(t)}\right)^{\frac{1}{2p-1}}.
\end{aligned}
\end{equation}
Separating of variables and integrating from $r$ to $R$, we obtain
\begin{equation*}
\begin{aligned}
&\left(\frac{A_f(x,R)}{A^a_H(R)}\right)^{\frac{1}{2p-1}}-\left(\frac{A_f(x,r)}{A^a_H(r)}\right)^{\frac{1}{2p-1}}\\
&\leq\left[\frac{n-1}{(2p-1)(2p-n)}\right]^{\frac{p-1}{2p-1}}
\left({\|{\mathrm{Ric}^H_f}_-\|_p}_{\,f,a}(R)\right)^{\frac{p}{2p-1}}
\cdot\int^R_r\left(\frac{1}{A_H(t)}\right)^{\frac{1}{2p-1}}dt.
\end{aligned}
\end{equation*}
Since the integral
\[
\int^R_r\left(\frac{1}{A_H(t)}\right)^{\frac{1}{2p-1}}dt\leq\int^R_0\left(\frac{1}{A_H(t)}\right)^{\frac{1}{2p-1}}dt
\]
converges when $p>n/2$, the conclusion \eqref{areacompineq} then follows.

\vspace{.1in}

For the case $H>0$ and $\frac{\pi}{2\sqrt{H}}<r\leq R<\frac{\pi}{\sqrt{H}}$, we have
\[
A^a_H(t)=e^{at}\left(\frac{\sin(\sqrt{H}t)}{\sqrt{H}}\right)^{n-1}.
\]
Then we use this function and \eqref{winteginesin2} instead of \eqref{wintegine2}
to get \eqref{areacompineqv2} by following the above similar argument.
\end{proof}

\vspace{.1in}

Using \eqref{keyinequ}, we can prove Theorem \ref{volcomp}, similar to the argument
of Petersen and Wei \cite{[PeWe]}.
\begin{proof}[Proof of Theorem \ref{volcomp}]
Using
\[
\frac{V_f(x,r)}{V^a_H(r)}=\frac{\int_0^rA_f(x,t)dt}{\int_0^rA^a_H(t)dt},
\]
we compute that
\begin{equation}\label{deri}
\frac{d}{dr}\left(\frac{V_f(x,r)}{V^a_H(r)}\right)
=\frac{A_f(x,r)\int_0^rA^a_H(t)dt-A^a_H(r)\int_0^rA_f(x,t)dt}{(V^a_H(r))^2}.
\end{equation}
On the other hand, integrating \eqref{keyinequ} from $t$ to $r$ ($t\le r$) gives
\begin{equation*}
\begin{aligned}
\frac{A_f(x,r)}{A^a_H(r)}&-\frac{A_f(x,t)}{A^a_H(t)}\\
&\leq C(n,p)\int_t^r
\frac{\left({\|{\mathrm{Ric}^H_f}_-\|_p}_{\,f,a}(s)\right)^{\frac{p}{2p-1}}}{A_H(s)^{\frac{1}{2p-1}}\cdot A^a_H(s)^{1-\frac{1}{2p-1}}}\cdot A_f(x,s)^{1-\frac{1}{2p-1}}ds\\
&\leq C(n,p)\frac{\left({\|{\mathrm{Ric}^H_f}_-\|_p}_{\,f,a}(r)\right)^{\frac{p}{2p-1}}}{A_H(t)^{\frac{1}{2p-1}}\cdot A^a_H(t)^{1-\frac{1}{2p-1}}}\cdot \int_t^rA_f(x,s)^{1-\frac{1}{2p-1}}ds\\
&\leq C(n,p)\frac{\left({\|{\mathrm{Ric}^H_f}_-\|_p}_{\,f,a}(r)\right)^{\frac{p}{2p-1}}}{A_H(t)^{\frac{1}{2p-1}}\cdot A^a_H(t)^{1-\frac{1}{2p-1}}}\cdot (r-t)^{\frac{1}{2p-1}}\,V_f(x,r)^{1-\frac{1}{2p-1}}.
\end{aligned}
\end{equation*}
This implies that
\begin{equation*}
\begin{aligned}
A_f&(x,r)A^a_H(t)-A^a_H(r)A_f(x,t)\\
&\leq C(n,p)\left({\|{\mathrm{Ric}^H_f}_-\|_p}_{\,f,a}(r)\right)^{\frac{p}{2p-1}}\cdot A^a_H(r)\cdot
e^{\frac{ar}{2p-1}}\cdot r^{\frac{1}{2p-1}}\,V_f(x,r)^{1-\frac{1}{2p-1}}.
\end{aligned}
\end{equation*}
Plugging this into \eqref{deri} gives
\begin{equation*}
\begin{aligned}
&\frac{d}{dr}\left(\frac{V_f(x,r)}{V^a_H(r)}\right)\\
&\leq C(n,p)\left({\|{\mathrm{Ric}^H_f}_-\|_p}_{\,f,a}(r)\right)^{\frac{p}{2p-1}}\cdot A^a_H(r)\cdot
e^{\frac{ar}{2p-1}}\cdot r^{\frac{2p}{2p-1}}\cdot \frac{V_f(x,r)^{1-\frac{1}{2p-1}}}{(V^a_H(r))^2}\\
&= C(n,p)\left({\|{\mathrm{Ric}^H_f}_-\|_p}_{\,f,a}(r)\right)^{\frac{p}{2p-1}}\cdot A_H(r)
\left(\frac{r\,e^{ar}}{V^a_H(r)}\right)^{\frac{2p}{2p-1}}\left(\frac{V_f(x,r)}{V^a_H(r)}\right)^{1-\frac{1}{2p-1}}.
\end{aligned}
\end{equation*}
Separating of variables and integrating from $r$ to $R$ ($r\le R$), we immediately get
\begin{equation*}
\begin{aligned}
&\left(\frac{V_f(x,R)}{V^a_H(R)}\right)^{\frac{1}{2p-1}}-\left(\frac{V_f(x,r)}{V^a_H(r)}\right)^{\frac{1}{2p-1}}\\
&\leq\left[\frac{n{-}1}{(2p{-}1)(2p{-}n)}\right]^{\frac{p-1}{2p-1}}\left({\|{\mathrm{Ric}^H_f}_-\|_p}_{\,f,a}(R)\right)^{\frac{p}{2p-1}}
\int^R_rA_H(t)\left(\frac{t\,e^{at}}{V^a_H(t)}\right)^{\frac{2p}{2p-1}}dt.
\end{aligned}
\end{equation*}
Since the integral
\begin{equation*}
\begin{aligned}
\int^R_rA_H(t)\left(\frac{t\,e^{at}}{V^a_H(t)}\right)^{\frac{2p}{2p-1}}dt
&\leq\int^R_rA_H(t)\left(\frac{t\,e^{at}}{V_H(t)}\right)^{\frac{2p}{2p-1}}dt\\
&\leq\int^R_0A_H(t)\left(\frac{t\,e^{at}}{V_H(t)}\right)^{\frac{2p}{2p-1}}dt
\end{aligned}
\end{equation*}
converges when $p>n/2$, the conclusion follows.
\end{proof}

As the classical case, the volume comparison estimate implies the volume
doubling estimate, which is often useful in various geometric inequalities.
\begin{corollary}[Volume doubling estimate]\label{corvde}
Let $(M,g,e^{-f}dv)$ be an $n$-dimensional smooth metric measure space.
Assume that $$\partial_rf\geq-a$$ for some constant $a\geq 0$, along
all minimal geodesic segments from $x\in M$. Given $\alpha>1$, $p>n/2$,
$H\in\mathbb{R}$ and $R>0$(assume $R\leq \frac{\pi}{2\sqrt{H}}$ when $H>0$),
there is an $\epsilon=\epsilon(n,p,aR,|H|R^2,\alpha)$ such that if
$R^2\cdot\bar{k}(p,H,a,R)<\epsilon$, then for all $x\in M$ and
$0<r_1<r_2\leq R$, we have
\[
\frac{V_f(x,r_2)}{V_f(x,r_1)}\leq \alpha\frac{V^a_H(r_2)}{V^a_H(r_1)}.
\]
\end{corollary}

\begin{remark}
We remark that $R^2\cdot\bar{k}(p,H,a,R)$ is the scale invariant curvature quantity.
Hence one can simply scale the metric so that one only need to work under the
assumption that $\bar{k}(p,H,a,1)$ is small.
\end{remark}

\begin{proof}[Proof of Corollary \ref{corvde}]
By Theorem \ref{volcomp}, we get
\begin{equation}\label{volcdineq}
\left(\frac{V_f(x,r_1)}{V_f(x,r_2)}\right)^{\frac{1}{2p-1}}\geq
\left(\frac{V^a_H(r_1)}{V^a_H(r_2)}\right)^{\frac{1}{2p-1}}(1-\sigma),
\end{equation}
where $\sigma:=C(n,p,H,a,r_2)\,{V^a_H(r_2)}^{\frac{1}{2p-1}}\cdot\bar{k}^{\frac{p}{2p-1}}(p,H,a,r_2)$.
Now we will estimate the quantity $(1-\sigma)$. We claim
that $\sigma(r)$ has some monotonicity in $r$ (though it is not really monotonic).
Indeed, since $C(n,p,H,a,r)$ is increasing in $r$, that is,
\[
\sigma(r_2){V^a_H(r_2)}^{-\frac{1}{2p-1}}\cdot\bar{k}^{-\frac{p}{2p-1}}(p,H,a,r_2)
\leq\sigma(R){V^a_H(R)}^{-\frac{1}{2p-1}}\cdot\bar{k}^{-\frac{p}{2p-1}}(p,H,a,R).
\]
By the definition of $\bar{k}$, the above inequality implies
\[
\sigma(r_2){V^a_H(r_2)}^{-\frac{1}{2p-1}}\cdot{V_f(x,r_2)}^{\frac{1}{2p-1}}
\leq\sigma(R){V^a_H(R)}^{-\frac{1}{2p-1}}\cdot{V_f(x,R)}^{\frac{1}{2p-1}}.
\]
Namely,
\begin{equation}\label{vovvpd}
\sigma(r_2)\leq\sigma(R)\left(\frac{V_f(x,R)}{V_f(x,r_2)}\right)^{\frac{1}{2p-1}}
\cdot\left(\frac{V^a_H(R)}{V^a_H(r_2)}\right)^{-\frac{1}{2p-1}}.
\end{equation}

On the other hand, by Theorem \ref{volcomp} again, we have
\[
\left(\frac{V_f(x,r_2)}{V^a_H(r_2)}\right)^{\frac{1}{2p-1}}\geq\left(\frac{V_f(x,R)}{V^a_H(R)}\right)^{\frac{1}{2p-1}}
\bigg[1-C(n,p,H,a,R){V^a_H(R)}^{\frac{1}{2p-1}}\cdot\bar{k}(R)^{\frac{p}{2p-1}}\bigg],
\]
where $\bar{k}(R)=\bar{k}(p,H,a,R)$ and
\[
C(n,p,H,a,R):=\left(\frac{n-1}{(2p-1)(2p-n)}\right)^{\frac{p-1}{2p-1}}
\int^R_0A_H(t)\left(\frac{t\,e^{at}}{V^a_H(t)}\right)^{\frac{2p}{2p-1}}dt.
\]
We also have
\[
C(n,p,H,a,R){V^a_H(R)}^{\frac{1}{2p-1}}\leq (e^{aR})^{\frac{2p+1}{2p-1}}R^{\frac{2p}{2p-1}}\,C(n,p,|H|R).
\]
Hence
\[
\left(\frac{V_f(x,r_2)}{V_f(x,R)}\right)^{\frac{1}{2p-1}}\geq\left(\frac{V^a_H(r_2)}{V^a_H(R)}\right)^{\frac{1}{2p-1}}
\bigg[1-C(n,p,|H|R)\cdot(e^{aR})^{\frac{2p+1}{2p-1}}\bigg(R^2\bar{k}(R)\bigg)^{\frac{p}{2p-1}}\bigg].
\]
When $R^2\bar{k}(R)\leq \epsilon$ is small enough, which depends only on $n$, $p$, $aR$ and $|H|R$,
the above inequality becomes
\[
\left(\frac{V_f(x,r_2)}{V_f(x,R)}\right)^{\frac{1}{2p-1}}\geq \frac 13\,\left(\frac{V^a_H(r_2)}{V^a_H(R)}\right)^{\frac{1}{2p-1}}.
\]
Substituting this into \eqref{vovvpd} yields
\[
\sigma(r_2)\leq 3\sigma(R).
\]
Combining this with \eqref{volcdineq} and letting $\sigma(R)$ arbitrary small
(as long as $R^2\bar{k}(R)\leq \epsilon$ is small enough), the result follows.
\end{proof}

\vspace{.1in}

In the rest of this section, we will study the relative volume comparison estimate
for annular regions and prove Theorem \ref{vocoannu} in the introduction. The proof
idea seems to be easy, by using the twice procedures of proving Theorem \ref{volcomp}.
\begin{proof}[Proof of Theorem \ref{vocoannu}]
On one hand, using
\[
\frac{V_f(x,r,R)}{V^a_H(r,R)}=\frac{\int_r^RA_f(x,t)dt}{\int_r^RA^a_H(t)dt},
\]
we have
\begin{equation}\label{derivat1}
\frac{d}{dR}\left(\frac{V_f(x,r,R)}{V^a_H(r,R)}\right)
=\frac{A_f(x,R)\int_r^RA^a_H(t)dt-A^a_H(R)\int_r^RA_f(x,t)dt}{(V^a_H(r,R))^2}.
\end{equation}
Integrating \eqref{keyinequ} from $t$ to $R$ ($t\leq R$) as before yields
\begin{equation*}
\begin{aligned}
&\frac{A_f(x,R)}{A^a_H(R)}-\frac{A_f(x,t)}{A^a_H(t)}\\
&\leq C(n,p)\frac{\left({\|{\mathrm{Ric}^H_f}_-\|_p}_{\,f,a}(R)\right)^{\frac{p}{2p-1}}}{A_H(t)^{\frac{1}{2p-1}}\cdot A^a_H(t)^{1-\frac{1}{2p-1}}}\cdot \int_t^RA_f(x,s)^{1-\frac{1}{2p-1}}ds\\
&\leq C(n,p)\frac{\left({\|{\mathrm{Ric}^H_f}_-\|_p}_{\,f,a}(R)\right)^{\frac{p}{2p-1}}}{A_H(t)^{\frac{1}{2p-1}}\cdot A^a_H(t)^{1-\frac{1}{2p-1}}}\cdot (R-t)^{\frac{1}{2p-1}}\,(V_f(x,t,R))^{1-\frac{1}{2p-1}},
\end{aligned}
\end{equation*}
which gives that
\begin{equation}
\begin{aligned}\label{differineq}
&A_f(x,R)A^a_H(t)-A^a_H(R)A_f(x,t)\\
&\leq C(n,p)\left({\|{\mathrm{Ric}^H_f}_-\|_p}_{\,f,a}(R)\right)^{\frac{p}{2p-1}}A^a_H(R)
e^{\frac{at}{2p-1}} R^{\frac{1}{2p-1}}\,(V_f(x,t,R))^{1-\frac{1}{2p-1}}.
\end{aligned}
\end{equation}
Substituting this into \eqref{derivat1},
\begin{equation*}
\begin{aligned}
\frac{d}{dR}\left(\frac{V_f(x,r,R)}{V^a_H(r,R)}\right)&\leq C(n,p)\left({\|{\mathrm{Ric}^H_f}_-\|_p}_{\,f,a}(R)\right)^{\frac{p}{2p-1}}\cdot A_H(R)\\
&\quad\times\left(\frac{R\,e^{aR}}{V^a_H(r,R)}\right)^{\frac{2p}{2p-1}}
\left(\frac{V_f(x,r,R)}{V^a_H(r,R)}\right)^{1-\frac{1}{2p-1}},
\end{aligned}
\end{equation*}
where we used the fact: $\int^R_rV_f(x,t,R)dt\le V_f(x,r,R)$. Separating of variables,
integrating with respect to the variable $R$ from $R_1$ to $R_2$ ($R_1\le R_2$), and
changing the variable $r$ to $r_2$ ($r_2\le R_1$), we get
\begin{equation*}
\begin{aligned}
&\left(\frac{V_f(x,r_2,R_2)}{V^a_H(r_2,R_2)}\right)^{\frac{1}{2p-1}}-\left(\frac{V_f(x,r_2,R_1)}{V^a_H(r_2,R_1)}\right)^{\frac{1}{2p-1}}\\
&\leq\left[\frac{n{-}1}{(2p{-}1)(2p{-}n)}\right]^{\frac{p-1}{2p-1}}\left({\|{\mathrm{Ric}^H_f}_-\|_p}_{\,f,a}(R_2)\right)^{\frac{p}{2p-1}}
\int^{R_2}_{R_1}A_H(t)\left(\frac{t\,e^{at}}{V^a_H(r_2,t)}\right)^{\frac{2p}{2p-1}}dt.
\end{aligned}
\end{equation*}

\vspace{.1in}

On the other hand, similar to the above argument, we also have
\begin{equation}\label{derivat2}
\frac{d}{dr}\left(\frac{V_f(x,r,R)}{V^a_H(r,R)}\right)
=\frac{A^a_H(r)\int_r^RA_f(x,t)dt-A_f(x,r)\int_r^RA^a_H(t)dt}{(V^a_H(r,R))^2}.
\end{equation}
By \eqref{differineq}, we also get that
\begin{equation*}
\begin{aligned}
&A_f(x,t)A^a_H(r)-A^a_H(t)A_f(x,r)\\
&\leq C(n,p)\left({\|{\mathrm{Ric}^H_f}_-\|_p}_{\,f,a}(t)\right)^{\frac{p}{2p-1}}A^a_H(t)\cdot
e^{\frac{ar}{2p-1}}\cdot t^{\frac{1}{2p-1}}\,(V_f(x,r,t))^{1-\frac{1}{2p-1}}
\end{aligned}
\end{equation*}
for $r\le t$. Substituting this into \eqref{derivat2}, and letting $R=R_1$, we have
\begin{equation*}
\begin{aligned}
\frac{d}{dr}\left(\frac{V_f(x,r,R_1)}{V^a_H(r,R_1)}\right)&\leq C(n,p)\left({\|{\mathrm{Ric}^H_f}_-\|_p}_{\,f,a}(R_1)\right)^{\frac{p}{2p-1}}\cdot A_H(R_1)\\
&\quad\times\left(\frac{R_1\,e^{aR_1}}{V^a_H(r,R_1)}\right)^{\frac{2p}{2p-1}}\left(\frac{V_f(x,r,R_1)}{V^a_H(r,R_1)}\right)^{1-\frac{1}{2p-1}}.
\end{aligned}
\end{equation*}
Separating of variables and integrating from $r_1$ to $r_2$ ($r_1\le r_2$) with respect to the variable $r$,
we immediately get
\begin{equation*}
\begin{aligned}
&\left(\frac{V_f(x,r_2,R_1)}{V^a_H(r_2,R_1)}\right)^{\frac{1}{2p-1}}-\left(\frac{V_f(x,r_1,R_1)}{V^a_H(r_1,R_1)}\right)^{\frac{1}{2p-1}}\\
&\leq\left[\frac{n{-}1}{(2p{-}1)(2p{-}n)}\right]^{\frac{p-1}{2p-1}}\left({\|{\mathrm{Ric}^H_f}_-\|_p}_{\,f,a}(R_1)\right)^{\frac{p}{2p-1}}
A_H(R_1)\int^{r_2}_{r_1}\left(\frac{R_1\,e^{aR_1}}{V^a_H(t,R_1)}\right)^{\frac{2p}{2p-1}}dt.
\end{aligned}
\end{equation*}
Combining the above two aspects,
\begin{equation*}
\begin{aligned}
&\left(\frac{V_f(x,r_2,R_2)}{V^a_H(r_2,R_2)}\right)^{\frac{1}{2p-1}}-\left(\frac{V_f(x,r_1,R_1)}{V^a_H(r_1,R_1)}\right)^{\frac{1}{2p-1}}\\
&\leq\left(\frac{n-1}{(2p-1)(2p-n)}\right)^{\frac{p-1}{2p-1}}\left({\|{\mathrm{Ric}^H_f}_-\|_p}_{\,f,a}(R_2)\right)^{\frac{p}{2p-1}}\\
&\quad\times\left[\int^{R_2}_{R_1}A_H(t)\left(\frac{t\,e^{at}}{V^a_H(r_2,t)}\right)^{\frac{2p}{2p-1}}dt+
\int^{r_2}_{r_1}A_H(R_1)\left(\frac{R_1\,e^{aR_1}}{V^a_H(t,R_1)}\right)^{\frac{2p}{2p-1}}dt\right]
\end{aligned}
\end{equation*}
for $0\leq r_1\leq r_2\leq R_1\leq R_2$. Hence the result follows.
\end{proof}

In particular, if $f$ is constant and $a=0$, we get volume comparison estimates for
the annuluses on Riemannian manifolds with integral bounds for the Ricci curvature.
\begin{corollary}\label{corvocoannu}
Let $(M,g)$ be an $n$-dimensional complete Riemannian manifold. Let
$H\in\mathbb{R}$ and $p>n/2$. For  $0\leq r_1\leq r_2\leq R_1\leq R_2$ (assume
$R_2\leq\frac{\pi}{2\sqrt{H}}$ when $H>0$), we have
\begin{equation*}
\begin{aligned}
&\left(\frac{V(x,r_2,R_2)}{V_H(r_2,R_2)}\right)^{\frac{1}{2p-1}}-\left(\frac{V(x,r_1,R_1)}{V_H(r_1,R_1)}\right)^{\frac{1}{2p-1}}\\
&\leq\left(\frac{n-1}{(2p-1)(2p-n)}\right)^{\frac{p-1}{2p-1}}\left(\|\mathrm{Ric}^H_-\|_p\,(R_2)\right)^{\frac{p}{2p-1}}\\
&\quad\times\left[\int^{R_2}_{R_1}A_H(t)\left(\frac{t}{V_H(r_2,t)}\right)^{\frac{2p}{2p-1}}dt+
\int^{r_2}_{r_1}A_H(R_1)\left(\frac{R_1}{V_H(t,R_1)}\right)^{\frac{2p}{2p-1}}dt\right].
\end{aligned}
\end{equation*}
\end{corollary}

\begin{remark}
\item[(1)]
If $r_1=r_2=0$, we immediately get the Petersen-Wei's relative Bishop-Gromov
volume comparison estimate in the integral sense \cite{[PeWe]}.
\item[(2)]
If $\mathrm{Ric}^H_- \equiv0$, i.e. $\mathrm{Ric}\ge(n-1)H$,
then we have a special case of the relative volume comparison estimate
for annuluses on manifolds (see \cite{[Zhu]}).
\end{remark}


\section{Mean curvature and volume comparison estimate II}\label{sec4}
In this section, we shall prove a very general mean curvature
comparison estimate on smooth metric measure spaces $(M,g,e^{-f}dv)$
when only the integral Bakry-\'{E}mery Ricci tensor bounds (without
any assumption on $f$), which might be useful in other applications.

\vspace{.1in}

In this case, we consider the following error form
\[
\psi:=(m_f-m_H)_+.
\]
Using this, we have
\begin{theorem}[Mean curvature comparison estimate II]\label{Mainthm2}
Let $(M,g,e^{-f}dv)$ be an $n$-dimensional smooth metric measure space.
Let $H\in\mathbb{R}$, and  $r\leq\frac{\pi}{2\sqrt{H}}$ when $H>0$.
For any $p>\frac n2$ when $n\geq 3$ ($p>\frac 54$ when $n=2$), we have
\begin{equation}\label{meanccom1}
\bigg(\int^r_0\sn_H^2(t)e^{\frac{4p-2}{n-1}f(t)}\psi(t)^{2p}\mathcal{A}_fdt\bigg)^{\frac 1p}
\leq\frac{2p-1}{2p-n} \Big(\mathcal{M}(r)+\mathcal{N}(r)\Big)
\end{equation}
and
\begin{equation}\label{meanccom2}
\sn_H^2(r)\psi(r)^{2p-1}e^{\frac{4p-2}{n-1}f(r)}\mathcal{A}_f
\leq\frac{(2p-1)^p}{(n-1)(2p-n)^{p-1}}\Big(\mathcal{M}(r)+\mathcal{N}(r)\Big)^p
\end{equation}
along any minimal geodesic segment from $x$, where
$$\mathcal{M}(r):=\bigg(\int^r_0\sn_H^2(t)e^{\frac{4p-2}{n-1}f(t)}m_H^{2p}\mathcal{A}_fdt\bigg)^{\frac 1p}$$
and
$$\mathcal{N}(r):=(n-1)\bigg(\int^r_0\sn_H^2(t)e^{\frac{4p-2}{n-1}f(t)}({\mathrm{Ric}^H_f}_-)^p\mathcal{A}_fdt\bigg)^{\frac 1p},$$
and where $m_H(r)=(n-1)\frac{\sn_H'(r)}{\sn_H(r)},$ and $\sn_H(r)$ is the unique function satisfying
\[
\sn_H''(r)+H\sn_H(r)=0,\quad\sn_H(0)=0,\quad\sn_H'(0)=1.
\]

\vspace{.1in}

Moreover, if $H>0$ and $\frac{\pi}{2\sqrt{H}}<r<\frac{\pi}{\sqrt{H}}$, then we have
\begin{equation}\label{meanccomposi}
\bigg(\int^r_0\sin^{4p-n-1}(\sqrt{H}t)\,e^{\frac{4p-2}{n-1}f(t)}\psi^{2p}\mathcal{A}_fdt\bigg)^{\frac 1p}\leq\frac{2p{-}1}{2p{-}n}\Big(\mathcal{\widetilde{M}}(r)+\mathcal{\widetilde{N}}(r)\Big)
\end{equation}
and
\begin{equation}\label{winteginesin2cc}
\sin^{4p-n-1}(\sqrt{H}r)\,e^{\frac{4p-2}{n-1}f(r)}\psi^{2p-1}\mathcal{A}_f
\leq(2p{-}1)^p\left(\frac{n{-}1}{2p{-}n}\right)^{p-1}\Big(\mathcal{\widetilde{M}}(r)+\mathcal{\widetilde{N}}(r)\Big)^p
\end{equation}
along that minimal geodesic segment from $x$, where
$$\mathcal{\widetilde{M}}(r):=\bigg(\int^r_0\sin^{4p-n-1}(\sqrt{H}t)\,e^{\frac{4p-2}{n-1}f(t)}m_H^{2p}\mathcal{A}_fdt\bigg)^{\frac 1p}$$
and
$$\mathcal{\widetilde{N}}(r):=(n-1)\bigg(\int^r_0\sin^{4p-n-1}(\sqrt{H}t)\,e^{\frac{4p-2}{n-1}f(t)}({\mathrm{Ric}^H_f}_-)^p\mathcal{A}_fdt\bigg)^{\frac 1p}.$$
\end{theorem}

It is unlucky that our theorem doesn't recover the classical case when the Ricci tensor has
pointwise lower bound and $f$ is constant. The main reason may be that we do not nicely
deal with the ``bad" term in the proof (see \eqref{keyinequ4} below). It is interesting to
know whether one has an improved estimate, which solves this problem.

\begin{proof}[Proof of Theorem \ref{Mainthm2}]
The proof's trick is partly inspired by the work of Wei-Wylie \cite{[WW]} and Petersen-Wei
\cite{[PeWe]}. Recall that,
\begin{equation*}
\begin{aligned}
(m_f-m_H)'\leq&-\frac{1}{n-1}[(m_f-m_H+\partial_r f)(m_f+m_H+\partial_r f)]+{\mathrm{Ric}^H_f}_-\\
=&-\frac{1}{n-1}\Big[(m_f-m_H)^2 +2(m_H+\partial_rf)(m_f-m_H)\\
&\qquad\qquad\qquad\qquad\qquad\qquad+\partial_r f(2m_H+\partial_r f)\Big]+{\mathrm{Ric}^H_f}_-.
\end{aligned}
\end{equation*}
Let $\psi:=(m_f-m_H)_+$. Then
\[
\psi'+\frac{\psi^2}{n-1}+\frac{2(m_H+\partial_rf)}{n-1}\psi
\leq-\frac{\partial_r f}{n-1}(2m_H+\partial_r f)+{\mathrm{Ric}^H_f}_-.
\]
When $\partial_r f=0$ and ${\mathrm{Ric}^H_f}_-=0$, we have $\psi=0$, and get the classical
mean curvature comparison. In general, notice that
\begin{equation}\label{keyinequ4}
-\frac{\partial_r f}{n-1}(2m_H+\partial_r f)=-\frac{(\partial_r f+m_H)^2}{n-1}+\frac{m^2_H}{n-1}
\le \frac{m^2_H}{n-1}.
\end{equation}
Therefore,
\[
\psi'+\frac{\psi^2}{n-1}+\frac{2(m_H+\partial_rf)}{n-1}\psi\leq\frac{m_H^2}{n-1}+{\mathrm{Ric}^H_f}_-.
\]
Multiplying this inequality by $(2p-1)\psi^{2p-2}\mathcal{A}_f$, we have
\begin{equation}
\begin{aligned}\label{keyineq1}
(2p-1)\psi^{2p-2}\psi'\mathcal{A}_f&+\frac{2p-1}{n-1}\psi^{2p}\mathcal{A}_f
+\frac{4p-2}{n-1}(m_H+\partial_rf)\psi^{2p-1}\mathcal{A}_f\\
&\quad\leq\frac{2p-1}{n-1}m_H^2\psi^{2p-2}\mathcal{A}_f+(2p-1){\mathrm{Ric}^H_f}_-\cdot\psi^{2p-2}\mathcal{A}_f.
\end{aligned}
\end{equation}
Notice that
\begin{equation*}
\begin{aligned}
(\psi^{2p-1} \mathcal{A}_f)' &=(2p-1)\psi^{2p-2}\psi'\,\mathcal{A}_f+\psi^{2p-1}\mathcal{A}'_f\\
&=(2p-1)\psi^{2p-2}\psi'\,\mathcal{A}_f+\psi^{2p-1} m_f\mathcal{A}_f.
\end{aligned}
\end{equation*}
So, \eqref{keyineq1} can be rewritten as
\begin{equation*}
\begin{aligned}
(\psi^{2p-1} \mathcal{A}_f)'-\psi^{2p-1}&m_f\mathcal{A}_f+\frac{2p-1}{n-1}\psi^{2p}\mathcal{A}_f
+\frac{4p-2}{n-1}(m_H+\partial_rf)\psi^{2p-1}\mathcal{A}_f\\
&\leq\frac{2p-1}{n-1}m_H^2\psi^{2p-2}\mathcal{A}_f+(2p-1){\mathrm{Ric}^H_f}_-\cdot\psi^{2p-2}\mathcal{A}_f.
\end{aligned}
\end{equation*}
Rearranging some terms by using $\psi:=(m_f-m_H)_+$, we have
\begin{equation*}
\begin{aligned}
(\psi^{2p-1} \mathcal{A}_f)'+\frac{2p-n}{n-1}&\psi^{2p}\mathcal{A}_f
+\left(\frac{4p-n-1}{n-1}m_H+\frac{4p-2}{n-1}\partial_rf\right)\psi^{2p-1}\mathcal{A}_f\\
&\leq\frac{2p-1}{n-1}m_H^2\psi^{2p-2}\mathcal{A}_f+(2p-1){\mathrm{Ric}^H_f}_-\cdot\psi^{2p-2}\mathcal{A}_f.
\end{aligned}
\end{equation*}
Multiplying this by the integrating factor  $\sn_H^2(r)e^{\frac{4p-2}{n-1} f(r)}$, we obtain
\begin{equation}
\begin{aligned}\label{zhongyao}
&\left[\sn_H^2(r) e^{\frac{4p-2}{n-1}f(r)}\psi^{2p-1} \mathcal{A}_f\right]'
+\frac{2p-n}{n-1}\sn_H^2(r)e^{\frac{4p-2}{n-1} f(r)}\psi^{2p}\mathcal{A}_f\\
&\qquad\qquad\qquad\qquad\qquad+\frac{4p-n-3}{n-1}\sn_H^2(r) e^{\frac{4p-2}{n-1}f(r)}m_H\psi^{2p-1}\mathcal{A}_f\\
&\leq\frac{2p-1}{n-1}\sn_H^2(r)e^{\frac{4p-2}{n-1} f(r)}m_H^2\psi^{2p-2}\mathcal{A}_f\\
&\qquad\qquad\qquad\qquad\qquad+(2p-1)\sn_H^2(r)e^{\frac{4p-2}{n-1}f(r)}{\mathrm{Ric}^H_f}_-\cdot\psi^{2p-2}\mathcal{A}_f,
\end{aligned}
\end{equation}
where we used $m_H(r)=(n-1)\frac{\sn_H'(r)}{\sn_H(r)}$. Since
$p>n/2$ when $n\geq 3$, and $p>5/4$ when $n=2$, then
\[
\frac{4p-n-3}{n-1}\sn_H^2(r)e^{\frac{4p-2}{n-1}f(r)} m_H\psi^{2p-1}\mathcal{A}_f\geq0.
\]
Hence we can throw away this term from the above inequality, and get
\begin{equation*}
\begin{aligned}
&\left[\sn_H^2(r) e^{\frac{4p-2}{n-1}f(r)}\psi^{2p-1}\mathcal{A}_f\right]'
+\frac{2p-n}{n-1}\sn_H^2(r)e^{\frac{4p-2}{n-1}f(r)}\psi^{2p}\mathcal{A}_f\\
&\leq\frac{2p{-}1}{n{-}1}\Bigg[\sn_H^2(r)e^{\frac{4p-2}{n-1} f(r)}m_H^2\psi^{2p-2}\mathcal{A}_f+(n{-}1)\sn_H^2(r)e^{\frac{4p-2}{n-1}f(r)}{\mathrm{Ric}^H_f}_-\cdot\psi^{n-1}\mathcal{A}_f\Bigg].
\end{aligned}
\end{equation*}
Since
\[
\sn_H^2(t)e^{\frac{4p-2}{n-1}f(t)}\psi^{2p-1}\mathcal{A}_f\Big|_{t=0}=0,
\]
integrating the above inequality from $0$ to $r$ yields
\begin{equation*}
\begin{aligned}
&\sn_H^2(r)e^{\frac{4p-2}{n-1}f(r)}\psi^{2p-1}\mathcal{A}_f+\frac{2p-n}{n-1}\int^r_0\sn_H^2(t)e^{\frac{4p-2}{n-1}f(t)}\psi^{2p}\mathcal{A}_fdt\\
&\leq\frac{2p-1}{n-1}\Bigg[\int^r_0\sn_H^2(t)e^{\frac{4p-2}{n-1}f(t)}m_H^2\psi^{2p-2}\mathcal{A}_fdt\\
&\qquad\qquad\quad+(n-1)\int^r_0\sn_H^2(t)e^{\frac{4p-2}{n-1}f(t)}{\mathrm{Ric}^H_f}_-\cdot\psi^{2p-2}\mathcal{A}_fdt\Bigg].
\end{aligned}
\end{equation*}
Since $p>n/2$, the first two terms of the above inequality are nonnegative. Hence,
\begin{equation}
\begin{aligned}\label{keyineu2}
&\sn_H^2(r)e^{\frac{4p-2}{n-1}f(r)}\psi^{2p-1}\mathcal{A}_f\\
&\leq\frac{2p-1}{n-1}\Bigg[\int^r_0\sn_H^2(t)e^{\frac{4p-2}{n-1}f(t)}m_H^2\psi^{2p-2}\mathcal{A}_fdt\\
&\qquad\qquad\quad+(n-1)\int^r_0\sn_H^2(t)e^{\frac{4p-2}{n-1}f(t)}{\mathrm{Ric}^H_f}_-\cdot\psi^{2p-2}\mathcal{A}_fdt\Bigg].
\end{aligned}
\end{equation}
and
\begin{equation}
\begin{aligned}\label{keyineu3}
&\frac{2p-n}{n-1}\int^r_0\sn_H^2(t)e^{\frac{4p-2}{n-1}f(t)}\psi^{2p}\mathcal{A}_fdt\\
&\leq\frac{2p-1}{n-1}\Bigg[\int^r_0\sn_H^2(t)e^{\frac{4p-2}{n-1}f(t)}m_H^2\psi^{2p-2}\mathcal{A}_fdt\\
&\qquad\qquad\quad+(n-1)\int^r_0\sn_H^2(t)e^{\frac{4p-2}{n-1}f(t)}{\mathrm{Ric}^H_f}_-\cdot\psi^{2p-2}\mathcal{A}_fdt\Bigg].
\end{aligned}
\end{equation}
By Holder inequality, we also have
\begin{equation}
\begin{aligned}\label{Holderina}
&\int^r_0\sn_H^2(t)e^{\frac{4p-2}{n-1}f(t)}m_H^2\psi^{2p-2}\mathcal{A}_fdt\\
&\leq\bigg[\int^r_0\sn_H^2(t)e^{\frac{4p-2}{n-1}f(t)}\psi^{2p} \mathcal{A}_fdt\bigg]^{1-\frac 1p}
\cdot\bigg[\int^r_0\sn_H^2(t)e^{\frac{4p-2}{n-1}f(t)}m_H^{2p} \mathcal{A}_fdt\bigg]^{\frac 1p}
\end{aligned}
\end{equation}
and
\begin{equation}
\begin{aligned}\label{Holderinb}
&\int^r_0\sn_H^2(t)e^{\frac{4p-2}{n-1}f(t)}{\mathrm{Ric}^H_f}_-\cdot\psi^{2p-2}\mathcal{A}_fdt\\
&\leq\bigg[\int^r_0\sn_H^2(t)e^{\frac{4p-2}{n-1}f(t)}\psi^{2p} \mathcal{A}_fdt\bigg]^{1-\frac 1p}
\cdot\bigg[\int^r_0\sn_H^2(t)e^{\frac{4p-2}{n-1}f(t)}({\mathrm{Ric}^H_f}_-)^p \mathcal{A}_fdt\bigg]^{\frac 1p}.
\end{aligned}
\end{equation}
Finally, combining \eqref{Holderina}, \eqref{Holderinb} and \eqref{keyineu3}, we obtain
\begin{equation}
\begin{aligned}\label{impoineq}
\bigg[\int^r_0\sn_H^2(t)&e^{\frac{4p-2}{n-1}f(t)}\psi^{2p}\mathcal{A}_fdt\bigg]^{\frac 1p}
\leq\frac{2p-1}{2p-n}\bigg[\int^r_0\sn_H^2(t)e^{\frac{4p-2}{n-1}f(t)}m_H^{2p} \mathcal{A}_fdt\bigg]^{\frac 1p}\\
&\qquad+\frac{(n-1)(2p-1)}{2p-n}\bigg[\int^r_0\sn_H^2(t)e^{\frac{4p-2}{n-1}f(t)}({\mathrm{Ric}^H_f}_-)^p \mathcal{A}_fdt\bigg]^{\frac 1p},
\end{aligned}
\end{equation}
which implies \eqref{meanccom1}. Combining \eqref{meanccom1}, \eqref{Holderina},
\eqref{Holderinb} and \eqref{keyineu2} yields \eqref{meanccom2}.

\vspace{.1in}

When $H>0$ and $\frac{\pi}{2\sqrt{H}}<r<\frac{\pi}{\sqrt{H}}$, we see that $m_H<0$ in
\eqref{zhongyao}. Similar to those discussion as before, multiplying by the
integrating factor $\sin^{4p-n-3}(\sqrt{H}r)$ in \eqref{zhongyao} and
integrating from $0$ to $r$, we get
\begin{equation}
\begin{aligned}\label{zhongyao2}
&\sin^{4p-n-1}(\sqrt{H}r)\cdot e^{\frac{4p-2}{n-1}f(r)}\psi^{2p-1} \mathcal{A}_f\\
&\qquad\qquad+\frac{2p-n}{n-1}\int^r_0\sin^{4p-n-1}(\sqrt{H}t)\cdot e^{\frac{4p-2}{n-1} f(t)}\psi^{2p}\mathcal{A}_f\,dt\\
&\leq\frac{2p-1}{n-1}\Bigg[\int^r_0\sin^{4p-n-1}(\sqrt{H}t)\cdot e^{\frac{4p-2}{n-1}f(t)}m_H^2\psi^{2p-2}\mathcal{A}_fdt\\
&\qquad\qquad\quad+(n-1)\int^r_0\sin^{4p-n-1}(\sqrt{H}t)\cdot e^{\frac{4p-2}{n-1}f(t)}{\mathrm{Ric}^H_f}_-\cdot\psi^{2p-2}\mathcal{A}_fdt\Bigg].
\end{aligned}
\end{equation}
Using Holder inequality as before we get
\begin{equation*}
\begin{aligned}
\bigg[&\int^r_0\sin^{4p-n-1}(\sqrt{H}t)e^{\frac{4p-2}{n-1}f(t)}\psi^{2p}\mathcal{A}_fdt\bigg]^{\frac 1p}\\
&\leq\frac{2p-1}{2p-n}\bigg[\int^r_0\sin^{4p-n-1}(\sqrt{H}t)e^{\frac{4p-2}{n-1}f(t)}m_H^{2p} \mathcal{A}_fdt\bigg]^{\frac 1p}\\
&\quad+\frac{(n-1)(2p-1)}{2p-n}\bigg[\int^r_0\sin^{4p-n-1}(\sqrt{H}t)e^{\frac{4p-2}{n-1}f(t)}({\mathrm{Ric}^H_f}_-)^p \mathcal{A}_fdt\bigg]^{\frac 1p},
\end{aligned}
\end{equation*}
which is \eqref{meanccomposi}. Finally we substitute \eqref{meanccomposi} into \eqref{zhongyao2} gives
\eqref{winteginesin2cc} by using the Holder inequality as before.
\end{proof}

\vspace{.1in}

In the following, we will apply mean curvature comparison estimate II to derive
another weighted volume comparison estimate in the integral sense. At first,
the weighted mean curvature comparison estimate II implies a tedious volume
comparison estimate of geodesic spheres.
\begin{theorem}\label{areacomp2}
Let $(M,g,e^{-f}dv)$ be an $n$-dimensional smooth metric measure space.
Let $H\in \mathbb{R}$ and $p>\frac n2$ when $n\geq 3$ ($p>\frac 54$ when $n=2$)
be given, and when $H>0$
assume that $R\leq \frac{\pi}{2\sqrt{H}}$. For $0<r\leq R$, we have
\begin{equation}
\begin{aligned}\label{areacompineqgen}
&\left(\frac{A_f(x,R)}{A_H(R)}\right)^{\frac{1}{2p-1}}-\left(\frac{A_f(x,r)}{A_H(r)}\right)^{\frac{1}{2p-1}}\\
&\leq C(n,p)\int^R_r\Big(\mathcal{M}(t)+\mathcal{N}(t)\Big)^{\frac{p}{2p-1}}
\,\sn_H^{-\frac{2}{2p-1}}(t)\,e^{-\frac{2f(t)}{n-1}}\,A_H^{-\frac{1}{2p-1}}(t)\,dt,
\end{aligned}
\end{equation}
where $$C(n,p):=\left(\frac{2p{-}n}{n{-}1}\right)^{\frac{1}{2p{-}1}}
\left(\frac{2p{-}1}{2p{-}n}\right)^{\frac{p}{2p{-}1}},$$
$$\mathcal{M}(t):=\bigg(\int^t_0\sn_H^2(s)\,e^{\frac{4p-2}{n-1}f(s)}m_H^{2p}\mathcal{A}_fds\bigg)^{\frac 1p},$$
and
$$\mathcal{N}(t):=(n-1)\bigg(\int^t_0\sn_H^2(s)\,e^{\frac{4p-2}{n-1}f(s)}({\mathrm{Ric}^H_f}_-)^p\mathcal{A}_fds\bigg)^{\frac 1p}.$$

In particular, if further assume $|f|\leq k$ for some constant $k\geq 0$; and
$\frac n2<p<\frac n2+1$ when $n\geq 3$ (when $n=2$, we assume $\frac 54<p<2$).
For $0<r\leq R$, we have
\begin{equation}
\begin{aligned}\label{areacompineq2}
&\left(\frac{A_f(x,R)}{A_H(R)}\right)^{\frac{1}{2p-1}}-\left(\frac{A_f(x,r)}{A_H(r)}\right)^{\frac{1}{2p-1}}\\
&\leq C(n,p)\,e^{\frac{4k}{n-1}}\,\bigg(\mathcal{P}(R)+\mathcal{Q}(R)\bigg)^{\frac{p}{2p-1}}
\int^R_r \sn_H^{-\frac{2}{2p-1}}(t)\,A_H^{-\frac{1}{2p-1}}(t)dt,
\end{aligned}
\end{equation}
where
$$\mathcal{P}(R):=\bigg(\int^R_0\sn_H^2(t)m_H^{2p}\mathcal{A}_fdt\bigg)^{\frac 1p}$$
and
$$\mathcal{Q}(R):=(n-1)\bigg(\int^R_0\sn_H^2(t)({\mathrm{Ric}^H_f}_-)^p\mathcal{A}_fdt\bigg)^{\frac 1p}.$$
\end{theorem}
\begin{remark}
We remark that $\mathcal{P}(R)$ converges when $\frac n2<p<\frac n2+1$, $n\geq 3$
(when $\frac 54<p<2$, $n=2$). However, for such $p$, if $r\to 0$, the integral $$\int^R_r \sn_H^{-\frac{2}{2p-1}}(t)\,A_H^{-\frac{1}{2p-1}}(t)dt$$ blows up.
\end{remark}

\begin{proof}[Proof of Theorem \ref{areacomp2}]
We apply
$\mathcal{A}'_f=m_f \mathcal{A}_f$ and $\mathcal{A}_H'=m_H\mathcal{A}_H$
to compute that
\[
\frac{d}{dt}\left(\frac{\mathcal{A}_f(t,\theta)}{\mathcal{A}_H(t)}\right)
=(m_f-m_H)\frac{\mathcal{A}_f(t,\theta)}{\mathcal{A}_H(t)}.
\]
Hence,
\begin{equation*}
\begin{aligned}
\frac{d}{dt}\left(\frac{A_f(x,t)}{A_H(t)}\right)
&=\frac{1}{Vol(S^{n-1})}\int_{S^{n-1}}\frac{d}{dt}\left(\frac{\mathcal{A}_f(t,\theta)}{\mathcal{A}_H(t)}\right)d\theta_{n-1}\\
&\leq\frac{1}{A_H(t)}\int_{S^{n-1}}\psi\cdot\mathcal{A}_f(t,\theta)d\theta_{n-1}.
\end{aligned}
\end{equation*}
Using Holder's inequality and \eqref{meanccom2}, we have
\begin{equation*}
\begin{aligned}
\int_{S^{n-1}}&\psi\cdot\mathcal{A}_f(t,\theta)d\theta_{n-1}\\
&\leq\left(\int_{S^{n-1}}\psi^{2p-1} \mathcal{A}_f(x,t)d\theta_{n-1}\right)^{\frac{1}{2p-1}}
A_f(x,t)^{1-\frac{1}{2p-1}}\\
&\leq C(n,p)\,\sn_H^{-\frac{2}{2p-1}}(t)\,e^{-\frac{2f(t)}{n-1}}\cdot\bigg(\mathcal{M}(t)+\mathcal{N}(t)\bigg)^{\frac{p}{2p-1}}
A_f(x,t)^{1-\frac{1}{2p-1}},
\end{aligned}
\end{equation*}
where $$C(n,p):=\left(\frac{2p-n}{n-1}\right)^{\frac{1}{2p-1}}
\,\left(\frac{2p-1}{2p-n}\right)^{\frac{p}{2p-1}},$$
$$\mathcal{M}(t):=\bigg(\int^t_0\sn_H^2(s)\,e^{\frac{4p-2}{n-1}f(s)}m_H^{2p}\mathcal{A}_fds\bigg)^{\frac 1p},$$
and
$$\mathcal{N}(t):=(n-1)\bigg(\int^t_0\sn_H^2(s)\,e^{\frac{4p-2}{n-1}f(s)}({\mathrm{Ric}^H_f}_-)^p\mathcal{A}_fds\bigg)^{\frac 1p}.$$
Hence,
\begin{equation}
\begin{aligned}\label{keyinequ2k}
\frac{d}{dt}\left(\frac{A_f(x,t)}{A_H(t)}\right)\leq C(n,p)&\,\sn_H^{-\frac{2}{2p-1}}(t)\,e^{-\frac{2f(t)}{n-1}}\cdot\left(\frac{A_f(x,t)}{A_H(t)}\right)^{1-\frac{1}{2p-1}}\\
&\times\bigg(\mathcal{M}(t)+\mathcal{N}(t)\bigg)^{\frac{p}{2p-1}}
\left(\frac{1}{A_H(t)}\right)^{\frac{1}{2p-1}}.
\end{aligned}
\end{equation}
Separating of variables and integrating from $r$ to $R$, we obtain
\begin{equation*}
\begin{aligned}
&\left(\frac{A_f(x,R)}{A_H(R)}\right)^{\frac{1}{2p-1}}-\left(\frac{A_f(x,r)}{A_H(r)}\right)^{\frac{1}{2p-1}}\\
&\leq\frac{C(n,p)}{2p-1}\,\int^R_r\Big(\mathcal{M}(t)+\mathcal{N}(t)\Big)^{\frac{p}{2p-1}}
\,\sn_H^{-\frac{2}{2p-1}}(t)\,e^{-\frac{2f(t)}{n-1}}\,A_H^{-\frac{1}{2p-1}}(t)dt.
\end{aligned}
\end{equation*}
Therefore we prove the first part of conclusions.

\vspace{.1in}

If we further assume $|f|\leq k$ for some constant $k\geq 0$, then
\begin{equation*}
\begin{aligned}
&\left(\frac{A_f(x,R)}{A_H(R)}\right)^{\frac{1}{2p-1}}-\left(\frac{A_f(x,r)}{A_H(r)}\right)^{\frac{1}{2p-1}}\\
&\leq\frac{C(n,p)}{2p-1}\,e^{\frac{4k}{n-1}}\,\bigg(\mathcal{P}(R)+\mathcal{Q}(R)\bigg)^{\frac{p}{2p-1}}
\int^R_r \sn_H^{-\frac{2}{2p-1}}(t)\,A_H^{-\frac{1}{2p-1}}(t)dt,
\end{aligned}
\end{equation*}
where
$$\mathcal{P}(R):=\bigg(\int^R_0\sn_H^2(t)m_H^{2p}\mathcal{A}_fdt\bigg)^{\frac 1p}$$
and
$$\mathcal{Q}(R):=(n-1)\bigg(\int^R_0\sn_H^2(t)({\mathrm{Ric}^H_f}_-)^p\mathcal{A}_fdt\bigg)^{\frac 1p}.$$
Notice that $\mathcal{P}(R)$ converges when $\frac n2<p<\frac n2+1$ if $n\geq 3$
(if $n=2$, we assume $\frac 54<p<2$). Hence the result follows.
\end{proof}

\vspace{.1in}

Similar to the first case of discussions (the case $\partial_r f\geq-a$), we can apply
\eqref{keyinequ2k} to obtain the following volume comparison estimate when
$f$ is bounded.
\begin{theorem}[Relative volume comparison estimate II]\label{volcomp2}
Let $(M,g,e^{-f}dv)$ be an $n$-dimensional smooth metric measure space.
Assume that
\[
|f(x)|\leq k
\]
for some constant $k\geq 0$. Let $H\in \mathbb{R}$ and $\frac n2<p<\frac n2+1$ when $n\geq 3$
(when $n=2$, we assume $\frac 54<p<2$)  be given, and when $H>0$
assume that $R\leq \frac{\pi}{2\sqrt{H}}$. For $0<r\leq R$, we have
\begin{equation}
\begin{aligned}\label{volcomp2eq}
&\left(\frac{V_f(x,R)}{V_H(R)}\right)^{\frac{1}{2p-1}}-\left(\frac{V_f(x,r)}{V_H(r)}\right)^{\frac{1}{2p-1}}\\
&\qquad\leq C(n,p)\,e^{\frac{4k}{n-1}}
\bigg(\mathcal{P}(R)+\mathcal{Q}(R)\bigg)^{\frac{p}{2p-1}}\int^R_rA_H(t)\left(\frac{t^{1-\frac 1p}}{V_H(t)}\right)^{\frac{2p}{2p-1}}dt.
\end{aligned}
\end{equation}
Here,
\[
C(n,p):=\frac{4p-2}{p-1}\left(\frac{(n-1)^{-\frac{1}{p-1}}}{(2p-1)(2p-n)}\right)^{\frac{p-1}{2p-1}},
\]
\[
\mathcal{P}(R):=\bigg(\int^R_0\sn_H^2(t)m_H^{2p}\mathcal{A}_fdt\bigg)^{\frac 1p}
\]
and
$$\mathcal{Q}(R):=(n-1)\bigg(\int^R_0\sn_H^2(t)({\mathrm{Ric}^H_f}_-)^p\mathcal{A}_fdt\bigg)^{\frac 1p}.$$
\end{theorem}
\begin{remark}
We remark that $\mathcal{P}(R)$ converges when $\frac n2<p<\frac n2+1$, $n\geq 3$ (when $\frac 54<p<2$, $n=2$).
However, for such $p$, if $r\to 0$, then the integral $$\int^R_rA_H(t)\left(\frac{t^{1-\frac 1p}}{V_H(t)}\right)^{\frac{2p}{2p-1}}dt$$ blows up.
\end{remark}

\begin{proof}[Proof of Theorem \ref{volcomp2}]
We use the formula
\[
\frac{V_f(x,r)}{V_H(r)}=\frac{\int_0^rA_f(x,t)dt}{\int_0^rA_H(t)dt},
\]
to compute that
\begin{equation}\label{deri2}
\frac{d}{dr}\left(\frac{V_f(x,r)}{V_H(r)}\right)
=\frac{A_f(x,r)\int_0^rA_H(t)dt-A_H(r)\int_0^rA_f(x,t)dt}{(V_H(r))^2}.
\end{equation}
On the other hand, integrating \eqref{keyinequ2k} from $t$ to $r$, and using the
Holder inequality, we get
\begin{equation*}
\begin{aligned}
&\frac{A_f(x,r)}{A_H(r)}-\frac{A_f(x,t)}{A_H(t)}\\
&\quad\leq C(n,p,k)\int_t^r
\frac{\big(\mathcal{P}(s)+\mathcal{Q}(s)\big)^{\frac{p}{2p-1}}}{\sn_H^{\frac{2}{2p-1}}(s)\cdot A_H(s)^{\frac{1}{2p-1}}}\cdot \left(\frac{A_f(x,s)}{A_H(s)}\right)^{1-\frac{1}{2p-1}}ds\\
&\quad\leq C(n,p,k)\frac{\big(\mathcal{P}(r)+\mathcal{Q}(r)\big)^{\frac{p}{2p-1}}}{A_H(t)}\cdot \int_t^r \sn_H^{-\frac{2}{2p-1}}(s)\cdot A_f(x,s)^{1-\frac{1}{2p-1}}ds\\
&\quad\leq C(n,p,k)\frac{\big(\mathcal{P}(r)+\mathcal{Q}(r)\big)^{\frac{p}{2p-1}}}{ A_H(t)}
\cdot\left(\int_t^r \sn_H^{-2}(s)ds\right)^{\frac{1}{2p-1}}\cdot V_f(x,r)^{1-\frac{1}{2p-1}}\\
&\quad\leq C(n,p,k)\frac{\big(\mathcal{P}(r)+\mathcal{Q}(r)\big)^{\frac{p}{2p-1}}}{ A_H(t)}
\cdot\left(\frac 4t\right)^{\frac{1}{2p-1}}\cdot V_f(x,r)^{1-\frac{1}{2p-1}},
\end{aligned}
\end{equation*}
where $$C(n,p,k):=\left(\frac{2p-n}{n-1}\right)^{\frac{1}{2p-1}}
\,\left(\frac{2p-1}{2p-n}\right)^{\frac{p}{2p-1}}\, e^{\frac{4k}{n-1}}.$$
This implies that
\begin{equation*}
\begin{aligned}
A_f&(x,r)A_H(t)-A_H(r)A_f(x,t)\\
&\leq 4C(n,p,k)\big(\mathcal{P}(r)+\mathcal{Q}(r)\big)^{\frac{p}{2p-1}}\cdot A_H(r)\cdot
t^{\frac{-1}{2p-1}}\,V_f(x,r)^{1-\frac{1}{2p-1}}.
\end{aligned}
\end{equation*}
Plugging this into \eqref{deri2} gives
\begin{equation*}
\begin{aligned}
&\frac{d}{dr}\left(\frac{V_f(x,r)}{V_H(r)}\right)\\
&\leq\frac{4(2p{-}1)}{2p{-}2}C(n,p,k)\big(\mathcal{P}(r){+}\mathcal{Q}(r)\big)^{\frac{p}{2p-1}}\,A_H(r)\cdot r^{\frac{2p-2}{2p-1}}\cdot \frac{V_f(x,r)^{1-\frac{1}{2p-1}}}{(V_H(r))^2}\\
&=\frac{4(2p{-}1)}{2p{-}2}C(n,p,k)\big(\mathcal{P}(r){+}\mathcal{Q}(r)\big)^{\frac{p}{2p-1}} A_H(r)
\left(\frac{r^{1{-}\frac 1p}}{V_H(r)}\right)^{\frac{2p}{2p-1}}\left(\frac{V_f(x,r)}{V_H(r)}\right)^{1{-}\frac{1}{2p-1}}.
\end{aligned}
\end{equation*}
Separating of variables and integrating from $r$ to $R$,
\begin{equation*}
\begin{aligned}
&\left(\frac{V_f(x,R)}{V_H(R)}\right)^{\frac{1}{2p-1}}-\left(\frac{V_f(x,r)}{V_H(r)}\right)^{\frac{1}{2p-1}}\\
&\leq\frac{2C(n,p,k)}{p-1}\,\big(\mathcal{P}(R)+\mathcal{Q}(R)\big)^{\frac{p}{2p-1}}\,
\int^R_rA_H(t)\left(\frac{t^{1-\frac 1p}}{V_H(t)}\right)^{\frac{2p}{2p-1}}dt.
\end{aligned}
\end{equation*}
Then the conclusion follows.
\end{proof}


\section{Applications of comparison estimates}\label{sec5}
In this section, we mainly apply mean curvature comparison estimate I and
volume comparison estimate I to prove the global diameter estimate, eigenvalue
upper estimate and the volume growth estimate when the normalized new
$L^p$-norm of Bakry-\'{E}mery Ricci tensor below $(n-1)H$ is small.

\vspace{.1in}

We first prove Theorem \ref{diam} in the introduction.
\begin{proof}[Proof of Theorem \ref{diam}]
Let $p_1$, $p_2$ are two points in $M$, and $x_0$ be a middle point between $p_1$ and $p_2$.
We also let $e(x)$ be the excess function for the points $p_1$ and $p_2$, i.e.
\[
e(x):=d(p_1,x)+d(p_2,x)-d(p_1,p_2).
\]
By the triangle inequality, we have
\[
e(x)\ge0\quad \mathrm{and}\quad e(x)\le 2r
\]
on a ball $B(x_0,r)$, where $r>0$. In the following, we will prove our result by
contradiction. That is, if there exist two points $p_1$, $p_2$ in $M$, such that
$d(p_1,p_2)>D$ for any sufficient large $D$, then we can show that excess function
$e$ is negative on $B(x_0,r)$, which is a contradiction. The detail discussion
is as follows.

By the mean curvature estimate \eqref{winteginesin}, by using a suitably large
comparison sphere (the radius is a little small than $\frac{\pi}{\sqrt{H}}$) we may
choose any large $D$ enough so that if $d(p_1,p_2)>D$,
then
\[
\Delta_f\,e\leq-K+\psi_1
\]
on $B(x_0,r)$, where $K$ is a large positive constant to be determined, and $\psi_1$
is an error term controlled by $C_1(n,p,a,H,r)\cdot\bar{k}(p,H,a,r)$.

Following the nice construction of Lemma 1.4 in Colding's paper in \cite{[Cold]}, let
$\Omega_j\subseteq B(x_0,r)$ be a sequence of smooth star-shaped domains
which converges to $B(x_0,r)-\mathrm{Cut}(x_0)$. Also let $u_i$ be a sequence of smooth
functions such that
\[
|u_i-e|<i^{-1},\quad|\nabla u_i|\leq 2+i^{-1},\quad \mathrm{and}
\quad\Delta_f\,u_i\le \Delta_f\,e+i^{-1}
\]
on $B(x_0,r)$. Set $h:=d^2(x_0,\cdot)-r^2$, and then $h$ is a negative smooth function
on $\Omega_j$. So by Green's formula with respect to the weighted measure $e^{-f}dv$, we have
\begin{equation*}
\int_{\Omega_j}(\Delta_f\,u_i) h-\int_{\Omega_j}u_i(\Delta_f\,h)
=\int_{\partial\Omega_j} h(\nu\,u_i)-\int_{\partial\Omega_j}u_i(\nu\,h),
\end{equation*}
where $\nu$ is the outward unit normal direction to $\Omega_j$. We notice that
\[
\Delta_f\,u_i\le-K+\psi_1+i^{-1}
\]
and
\begin{equation*}
\begin{aligned}
u_i(\Delta_f\,h)&\leq(e+i^{-1})(2d\,\Delta_f\,d+2)\\
&\leq(e+i^{-1})(2n+2ad+\psi_2)\\
&\leq3r(2n+2ar+\psi_2),
\end{aligned}
\end{equation*}
where $\psi_2$ is another error term still controlled by $C_2(n,p,a,H,r)\cdot\bar{k}(p,H,a,r)$.
Therefore, we have
\begin{equation*}
\int_{\Omega_j}(-K+\psi_1+i^{-1})h-3r\int_{\Omega_j}(2n+2ar+\psi_2)
\leq\int_{\partial\Omega_j} h(2+i^{-1})-\int_{\partial\Omega_j}u_i(\nu\,h).
\end{equation*}
Since $u_i\to e$ when $i\to\infty$, by the dominated convergence theorem, the above
inequality implies
\begin{equation}\label{useineq}
\int_{\Omega_j}(-K+\psi_1)h-3r\int_{\Omega_j}(2n+2ar+\psi_2)
\leq2\int_{\partial\Omega_j} h-\int_{\partial\Omega_j}e(\nu\,h).
\end{equation}

We also notice that
\begin{equation*}
\begin{aligned}
\int_{B(x_0,r)}(-K+\psi_1)h&\ge\int_{B(x_0,\frac{3}{4}r)}-Kh-\int_{B(x_0,r)}r^2\psi_1\\
&\ge\int_{B(x_0,\frac{3}{4}r)}\frac{7}{16}r^2K-\int_{B(x_0,r)}r^2\psi_1\\
&=\frac{7}{16}r^2K\cdot V_f\left(x_0,\frac{3}{4}r\right)-\int_{B(x_0,r)}r^2\psi_1
\end{aligned}
\end{equation*}
and hence,
\begin{equation*}
\begin{aligned}
&\int_{B(x_0,r)}(-K+\psi_1)h-3r\int_{B(x_0,r)}(2n+2ar+\psi_2)\\
&\geq\frac{7}{16}r^2K\cdot V_f\left(x_0,\frac{3}{4}r\right)-6(nr+ar^2)V_f(x_0,r)
-\int_{B(x_0,r)}(r^2\psi_1+3r\psi_2).
\end{aligned}
\end{equation*}
By relative volume comparison estimate, if $\bar{k}(p,r,H,a)$ is small enough,
then
\[
V_f\left(x_0,\frac{3}{4}r\right)\geq 2^{-1}e^{-ar}\left(\frac{\sin(\frac{3}{4}r)}{\sin r}\right)^n\cdot V_f(x_0,r).
\]
In fact, at this case, since $r\to\left(\frac{\pi}{\sqrt{H}}\right)-$, we have the
following comparison
\[
\left(\frac12\right)^{4p-n-1}\varphi^{2p-1}\mathcal{A}_f\, e^{-ar}
\leq(2p-1)^p\left(\frac{n-1}{2p-n}\right)^{p-1}\cdot\int^r_0({\mathrm{Ric}^H_f}_-)^p \mathcal{A}_fe^{-at}dt.
\]
Moreover, if $\bar{k}(p,H,a,r)$ is small enough, we also have
\[
\int_{B(x_0,r)}(r^2\psi_1+3r\psi_2)\leq (nr+r^2)V_f(x_0,r).
\]
Thus, if we choose
\[
K>\frac{32}{7}(7nr^{-1}+6a+1)e^{ar}\left(\frac{\sin r}{\sin(\frac r2)}\right)^n,
\]
then
\[
\int_{B(x_0,r)}(-K+\psi_1)h-3r\int_{B(x_0,r)}(2n+2ar+\psi_2)>0.
\]
Combining this and \eqref{useineq} immediately yields
\[
2\int_{\partial\Omega_j} h-\int_{\partial\Omega_j}e(\nu\,h)>0
\]
as $j\to \infty$. However, the first integral of the above inequality goes to zero
as $j\to \infty$; while in the second integral of the above inequality: $\nu\,h\geq 0$
on $\partial\Omega_j$ for all $j$, as $\Omega_j$ is star-shaped. This forces that
the excess function $e$ must be negative on $B(x_0,r)$, which is a contradiction
to the fact: $e\geq 0$. Hence $d(p_1,p_2)\leq D$ for some $D$.
\end{proof}

\vspace{.1in}

Next we apply the similar argument of Petersen-Sprouse \cite{[PeSp]} to prove
Theorem \ref{eigen} in the introduction.
\begin{proof}[Proof of Theorem \ref{eigen}]
The proof is easy only by some direct computation. Recall that $B(\bar{x}_0,R)$
is a metric ball in the weighted model space $M^n_{H,a}$, where
$R\leq \frac{\pi}{2\sqrt{H}}$. Let $\lambda^D_1(n,H,a,R)$ be the first
eigenvalue of the $h$-Laplacian $\Delta_h$ with the Dirichlet condition in
$M^n_{H,a}$, where $h(x):=-a\cdot d(\bar{x}_0,x)$,
and $u(x)=\phi(r)$ be the corresponding eigenfunction, which satisfies
\[
\phi''+(m_H+a)\phi'+\lambda^D_1(n,H,a,R)\phi=0,\quad\phi(0)=1,\quad\phi(R)=0.
\]
It is easy to see that $0\leq\phi\leq 1$, since $\phi'<0$ on $[0,R]$. Now
we consider the Rayleigh quotient of the function $u(x)=\phi(d(x_0, x))$.
In the course of the proof, we will use the relative volume comparison estimate
when volume normalization of some integral Bakry-\'{E}mery Ricci tensor is
sufficient small. Now, a direct computation yields that
\begin{equation*}
\begin{aligned}
\int_{B(x_0,R)}|\nabla u|^2 e^{-f}dv&=\int_{S^{n-1}}\int^R_0(\phi')^2\mathcal{A}_f(t,\theta)\,dtd\theta_{n-1}\\
&=\int_{S^{n-1}}\left(\phi\phi'\mathcal{A}_f\big|^R_0-\int^R_0\phi(\phi'\mathcal{A}_f)'\,dt\right)d\theta_{n-1}\\
&=-\int_{S^{n-1}}\int^R_0\phi(\phi''+m_f\phi')\mathcal{A}_f\,dtd\theta_{n-1}\\
&=-\int_{S^{n-1}}\int^R_0\phi(\phi''+(m_H+a)\phi')\mathcal{A}_f\,dtd\theta_{n-1}\\
&\quad-\int_{S^{n-1}}\int^R_0(m_f-m_H-a)\phi\phi'\mathcal{A}_f\,dtd\theta_{n-1}\\
&\leq\lambda^D_1(n,H,a,R)\int_{S^{n-1}}\int^R_0\phi^2\mathcal{A}_f\,dtd\theta_{n-1}\\
&\quad+\int_{S^{n-1}}\int^R_0(m_f-m_H-a)_+|\phi'|\mathcal{A}_f\,dtd\theta_{n-1}.
\end{aligned}
\end{equation*}
Hence the Rayleigh quotient satisfies
\begin{equation*}
\begin{aligned}
Q&=\frac{\int_{B(x_0,R)}|\nabla u|^2 e^{-f}dv}{\int_{B(x_0,R)}u^2 e^{-f}dv}\\
&\leq\lambda^D_1(n,H,a,R)+\frac{\int_{S^{n-1}}\int^R_0(m_f{-}m_H{-}a)_+\,|\phi'|\mathcal{A}_f\,dtd\theta_{n-1}}
{\int_{S^{n-1}}\int^R_0\phi^2\mathcal{A}_f\,dtd\theta_{n-1}}.
\end{aligned}
\end{equation*}
Now choose the first value $r=r(n,H,a,R)$ such that $\phi(r)=1/2$. Then the last error term can be estimated:
\begin{equation*}
\begin{aligned}
&\frac{\int_{S^{n-1}}\int^R_0(m_f{-}m_H{-}a)_+\,|\phi'|\mathcal{A}_f}
{\int_{S^{n-1}}\int^R_0\phi^2\mathcal{A}_f}\\
&\quad\leq\frac{\Big(\int_{S^{n-1}}\int^R_0(m_f{-}m_H{-}a)^2_+\,\mathcal{A}_f\Big)^{\frac 12}
\cdot\Big(\int_{S^{n-1}}\int^R_0|\phi'|^2\mathcal{A}_f\Big)^{\frac 12}}
{\frac 12 V^{\frac 12}_f(x_0,r)\cdot\Big(\int_{S^{n-1}}\int^R_0\phi^2\mathcal{A}_f\Big)^{\frac 12}}\\
&\quad\leq2\left(\frac{\int_{S^{n-1}}\int^R_0(m_f{-}m_H{-}a)^2_+\,\mathcal{A}_f}{V_f(x_0,r)}\right)^{\frac 12}
\sqrt{Q}.
\end{aligned}
\end{equation*}
On the other hand, if $\bar{k}(p,H,a,R)$ is very small, then we have the following
volume doubling estimate (see Corollary \ref{corvde}):
\[
\frac{V_f(x_0,R)}{V_f(x_0,r)}\leq 4\frac{V^a_H(R)}{V^a_H(r)}.
\]
Putting this into the above error estimate, we have
\begin{equation*}
\begin{aligned}
&\frac{\int_{S^{n-1}}\int^R_0(m_f{-}m_H{-}a)_+\,|\phi'|\mathcal{A}_f}
{\int_{S^{n-1}}\int^R_0\phi^2\mathcal{A}_f}\\
&\quad\leq4\left(\frac{V^a_H(R)}{V^a_H(r)}\right)^{1/2}
\left(\frac{\int_{S^{n-1}}\int^R_0(m_f{-}m_H{-}a)^2_+\,\mathcal{A}_f}{V_f(x_0,R)}\right)^{\frac 12}\sqrt{Q}.
\end{aligned}
\end{equation*}
By the Holder inequality, we observe that
\begin{equation*}
\begin{aligned}
\int_{S^{n-1}}&\int^R_0(m_f{-}m_H{-}a)^2_+\,\mathcal{A}_f\leq
\left(\int_{S^{n-1}}\int^R_0(m_f{-}m_H{-}a)^{2p}_+\,\mathcal{A}_f e^{-at}\right)^{\frac 1p}\\
&\qquad\qquad\qquad\qquad\qquad\quad\times\left(\int_{S^{n-1}}\int^R_0\mathcal{A}_f e^{at}\right)^{1-\frac 1p}\\
&\leq e^{a(1-\frac 1p)R}\cdot V_f(x_0,R)^{1-\frac 1p}
\left(\int_{S^{n-1}}\int^R_0(m_f{-}m_H{-}a)^{2p}_+\,\mathcal{A}_f e^{-at}\right)^{\frac 1p}.
\end{aligned}
\end{equation*}
Using this, we further have
\begin{equation*}
\begin{aligned}
&\frac{\int_{S^{n-1}}\int^R_0(m_f{-}m_H{-}a)_+\,|\phi'|\mathcal{A}_f}
{\int_{S^{n-1}}\int^R_0\phi^2\mathcal{A}_f}\\
&\quad\leq4 e^{\frac{(p-1)a}{2p}R}\left(\frac{V^a_H(R)}{V^a_H(r)}\right)^{\frac 12}
\left(\frac{\int_{S^{n-1}}\int^R_0(m_f{-}m_H{-}a)^{2p}_+\,\mathcal{A}_fe^{-at}}{V_f(x_0,R)}\right)^{\frac {1}{2p}}\sqrt{Q}\\
&\quad\leq C(n,p,H,a,R)(\bar{k}(p,H,a,R))^{\frac 12}\sqrt{Q}.
\end{aligned}
\end{equation*}
Therefore,
\[
Q\leq\lambda^D_1(n,H,a,R)+C(n,p,H,a,R)(\bar{k}(p,H,a,R))^{\frac {1}{2p}}\sqrt{Q},
\]
which implies the desired result.
\end{proof}

\

Finally, we use Theorem \ref{vocoannu} to prove Theorem \ref{volumegrow}. The proof method
is similar to the classical case.
\begin{proof}[Proof of Theorem \ref{volumegrow}]
Since $a=0$ and $H=0$, Theorem \ref{vocoannu} in fact can be simply written as
\[
\left(\frac{V_f(x,r_2,R_2)}{V_0(r_2,R_2)}\right)^{\frac{1}{2p-1}}
-\left(\frac{V_f(x,r_1,R_1)}{V_0(r_1,R_1)}\right)^{\frac{1}{2p-1}}
\leq \mathcal{C}\cdot\left({\|{\mathrm{Ric}^0_f}_-\|^p_p}_{\,f,0}(R_2)\right)^{\frac{1}{2p-1}},
\]
for any $0\leq r_1\leq r_2\leq R_1\leq R_2$, where
\begin{equation*}
\begin{aligned}
\mathcal{C}&:=C(n,p)\left[\int^{r_2}_{r_1}R^{n-1}_1\left(\frac{R_1}{(R_1-t)^n}\right)^{\frac{2p}{2p-1}}dt
+\int^{R_2}_{R_1}t^{n-1}\left(\frac{t}{(t-r_2)^n}\right)^{\frac{2p}{2p-1}}dt\right]\\
&\leq2C(n,p)\, R_2^n\,\left(\frac{R_2}{(R_1-r_2)^n}\right)^{\frac{2p}{2p-1}}.
\end{aligned}
\end{equation*}
Let $x\in M$ be a point with $d(x_0,x)=R\geq 2$. Letting $r_1=0$,
$r_2=R-1$, $R_1=R$ and $R_2=R+1$ in the above inequality, then
\begin{equation*}
\begin{aligned}
&\frac{V_f(x,R+1)-V_f(x,R-1)}{(R+1)^n-(R-1)^n}\\
&\leq\left[\left(\frac{V_f(x,R)}{R^n}\right)^{\frac{1}{2p-1}}+2C(n,p)(R{+}1)^{n{+}\frac{2p}{2p-1}}\left({\|{\mathrm{Ric}^0_f}_-\|^p_p}_{\,f,0}(R{+}1)\right)^{\frac{1}{2p-1}}\right]^{2p-1}.
\end{aligned}
\end{equation*}
Using the inequality $(a+b)^m\leq 2^{m-1}(a^m+b^m)$ for all $a>0$ and $b>0$ with $m=2p-1$,
we have the following inequality
\begin{equation*}
\begin{aligned}
&\frac{V_f(x,R+1)-V_f(x,R-1)}{(R+1)^n-(R-1)^n}\\
&\leq \tilde{C}(n,p)\frac{V_f(x,R)}{R^n}+\tilde{C}(n,p)(R+1)^{n(2p-1)+2p}\,{\|{\mathrm{Ric}^0_f}_-\|^p_p}_{\,f,0}(R+1)
\end{aligned}
\end{equation*}
for some constant $\tilde{C}(n,p)$. Multiplying this inequality by $\frac{(R+1)^n-(R-1)^n}{V_f(x,R+1)}$,
by the definition of $\bar{k}$, we hence have
\begin{equation*}
\begin{aligned}
&\frac{V_f(x,R+1)-V_f(x,R-1)}{V_f(x,R+1)}\\
&\leq \frac{D(n,p)}{R}+D(n,p)(R+1)^{n(2p-1)+2p}\cdot\bar{k}^p(p,0,0,R+1).
\end{aligned}
\end{equation*}
for some constant $D(n,p)$. Now we choose $\epsilon=\epsilon(n,p,R)$ small enough with
$\bar{k}(p,0,0,R+1)<\epsilon$, such that
\begin{equation*}
\begin{aligned}
&\frac{V_f(x,R+1)-V_f(x,R-1)}{V_f(x,R+1)}\leq \frac{2D(n,p)}{R}.
\end{aligned}
\end{equation*}
Since $B(x_0,1)\subset B(x,R+1)\setminus B(x,R-1)$ and $B(x, R+1)\subset B(x_0,2R+1)$, hence
we have
\[
V_f(x_0,2R+1)\geq V_f(x,R+1)\geq \frac{V_f(x_0,1)}{2D(n,p)}\,R
\]
for the $R\geq2$.
\end{proof}


\section{Appendix. Comparison estimates for integral bounds of $m$-Bakry-Emery Ricci tensor}

In this section, we will state $f$-mean curvature comparison estimates and
relative $f$-volume comparison estimates when only the weighted integral bounds of
the $m$-Bakry-\'{E}mery Ricci tensor. Since the proof is almost the same as
the manifold case, we omit these proofs here.

Recall that another natural generalization of the Ricci tensor associated to
smooth metric measure space $(M,g,e^{-f}dv_g)$ is called $m$-Bakry-\'Emery Ricci
tensor, which is defined by
\[
\mathrm{Ric}_f^m:=\mathrm{Ric}_f-\frac{1}{m}df\otimes df
\]
for some number $m>0$. This curvature tensor is also introduced by Bakry and \'Emery
\cite{[BE]}. Here $m$ is finite, and we have the Bochner formula for the $m$-Bakry-\'Emery
Ricci tensor
\begin{equation}\label{weighboch}
\begin{aligned}
\frac 12\Delta_f|\nabla u|^2&=|\mathrm{Hess}u|^2
+\langle\nabla\Delta_f u, \nabla u\rangle+\mathrm{Ric}_f(\nabla u, \nabla u)\\
&\geq\frac{(\Delta_f u)^2}{m+n}
+\langle\nabla\Delta_f u, \nabla u\rangle+\mathrm{Ric}_f^m(\nabla u, \nabla u)
\end{aligned}
\end{equation}
for some $u\in C^\infty(M)$, which is regarded as the Bochner formula of the Ricci
curvature of an $(n+m)$-dimensional manifold. Hence many classical geometrical
and topological results for manifolds with Ricci tensor bounded below can be
easily extended to smooth metric measure spaces with $m$-Bakry-\'Emery Ricci
tensor bounded below (without any assumption on $f$), see for example
\cite{[BQ1],[BQ2],[LD],[WW0]} for details.

Let $(M,g,e^{-f}dv_g)$ be an $n$-dimensional smooth metric measure space.
For each $x\in M$, $m>0$ and let $\lambda\left( x\right) $ denote the smallest
eigenvalue for the tensor $\mathrm{Ric}^m_f:T_xM\to T_xM$. We define
\[
{\mathrm{Ric}^{m\,H}_f}_-:=\big((n+m-1)H-\lambda(x)\big)_+,
\]
where $H\in \mathbb{R}$, which measures the amount of $m$-Bakry-\'Emery
Ricci tensor below $(n+m-1)H$. We also introduce a
$L_f^p$-norm of function $\phi$, with respect to the weighted measure
$e^{-f}dv_g$:
\[
{\|\phi\|_p}_f(r):=\sup_{x\in M}\Big(\int_{B_x(r)}|\phi|^p\cdot e^{-f}dv_g\Big)^{\frac 1p}.
\]
Clearly, ${\|{\mathrm{Ric}^{m\,H}_f}_-\|_p}_f(r)=0$ iff ${\mathrm{Ric}^m_f}\geq (n+m-1)H$.
Notice that when $f$ is constant, all above notations are just as the usual quantities on
manifolds.

Let $r(y) = d(y, x)$ be the distance function from $x$ to $y$, and define
\[
\varphi:=(\Delta_f\,r-m^{n+m}_H)_+,
\]
the error from weighted mean curvature comparison in \cite{[BQ2]}. Here
$m^{n+m}_H$ denotes the mean curvature of the geodesic sphere in $M^{n+m}_H$, the
$n+m$-dimensional simply connected space with constant sectional curvature $H$.
The weighted Laplacian comparison states that if ${\mathrm{Ric}^m_f}\geq (n+m-1)H$,
then $\Delta_f\,r \le m^{m+n}_H$ (see for example \cite{[BQ2], [WW0]}). Using
\eqref{weighboch}, following the discussion in Section \ref{sec2}, we can similarly
generalize Petersen-Wei's and Aubry's comparison results to the case of smooth
metric measure spaces with only the $m$-Bakry-\'Emery Ricci tensor integral bounds.

\begin{theorem}\label{Mainthman}
Let $(M,g,e^{-f}dv)$ be an $n$-dimensional smooth metric measure space.
For any $p>\frac{n+m}{2}$, $m>0$, $H\in\mathbb{R}$ and $r>0$ (assume
$r\leq\frac{\pi}{2\sqrt{H}}$ when $H>0$), then
\[
{\|\varphi\|_{2p}}_f(r)\leq \left[\frac{(n+m-1)(2p-1)}{2p-n-m}
{\|{\mathrm{Ric}^{m\,H}_f}_-\|_p}_f(r)\right]^{\frac 12}
\]
and
\[
\varphi^{2p-1}\mathcal{A}_f
\leq(2p-1)^p\left(\frac{n+m-1}{2p-n-m}\right)^{p-1}\cdot\int^r_0({\mathrm{Ric}^{m\,H}_f}_-)^p \mathcal{A}_fdt
\]
along any minimal geodesic segment from $x$.

\vspace{.1in}

Moreover, if $H>0$ and $\frac{\pi}{2\sqrt{H}}<r<\frac{\pi}{\sqrt{H}}$, then
\[
{\Big\|\sin^{\frac{4p-n-m-1}{2p}}(\sqrt{H}t)\cdot\varphi\Big\|_{2p}}_f(r)
\leq\left[\frac{(n+m-1)(2p-1)}{2p-n-m}
{\|{\mathrm{Ric}^{m\,H}_f}_-\|_p}_f(r)\right]^{\frac 12}
\]
and
\[
\sin^{4p-n-m-1}(\sqrt{H}r)\varphi^{2p-1}\mathcal{A}_f
\leq(2p-1)^p\left(\frac{n+m-1}{2p-n-m}\right)^{p-1}
\cdot\int^r_0({\mathrm{Ric}^{m\,H}_f}_-)^p \mathcal{A}_fdt
\]
along any minimal geodesic segment from $x$.
\end{theorem}

\vspace{.1in}

Using Theorem \ref{Mainthman}, we have the corresponding volume
comparison estimate when only the weighted integral bounds of ${\mathrm{Ric}^m_f}$.

\begin{theorem}\label{volcomp2}
Let $(M,g,e^{-f}dv)$ be an $n$-dimensional smooth metric measure space. Let
$H\in\mathbb{R}$ and $p>\frac{n+m}{2}$, $m>0$. For $0<r\leq R$ (assume
$R\leq\frac{\pi}{2\sqrt{H}}$ when $H>0$),
\[
\left(\frac{V_f(x,R)}{V^{n+m}_H(R)}\right)^{\frac{1}{2p-1}}-\left(\frac{V_f(x,r)}{V^{n+m}_H(r)}\right)^{\frac{1}{2p-1}}
\leq C(n,m,p,H,R)\Big({\|{\mathrm{Ric}^H_f}_-\|^p_p}_f(R)\Big)^{\frac{1}{2p-1}}.
\]
Here,
\[
C(n,m,p,H,R):=\left(\frac{n{+}m{-}1}{(2p{-}1)(2p{-}n{-}m)}\right)^{\frac{p-1}{2p-1}}
\int^R_0A^{n+m}_H(t)\left(\frac{t}{V^{n+m}_H(t)}\right)^{\frac{2p}{2p-1}}dt,
\]
where $V^{n+m}_H(t)=\int^t_0A^{n+m}_H(s)ds$, $A^{n+m}_H(t)=\int_{S^{n-1}}\mathcal{A}^{n+m}_H(t,\theta)d\theta_{n-1}$,
and $\mathcal{A}^{n+m}_H$ is the volume element in the model space $M^{n+m}_H$.
\end{theorem}

\vspace{.1in}
Similar to the manifolds case, comparison estimates for the weighted integral
bounds of $\mathrm{Ric}^m_f$ have many applications, which will be not treated
here. Recall that \eqref{weighboch} allows one to extend many classical results for
manifolds of pointwise Ricci tensor condition to smooth metric measure spaces of
pointwise ${\mathrm{Ric}^m_f}$ condition, such as \cite{[BQ1],[BQ2],[LD],[WW0]}. In a
similar fashion, because Theorems \ref{Mainthman} and \ref{volcomp2} for $n$-dimensional
smooth metric measure spaces are essentially the same as the usual $(n+m)$-manifolds
case. We believe that many geometrical and topological results for the integral
Ricci tensor, such as Myers' type theorems \cite{[Au],[PeSp]}, finiteness fundamental
group theorems \cite{[Au],[HuXu]}, the first Betti number estimate \cite{[HuXu]},
Gromov's bounds on the volume entropy \cite{[Au2]}, compactness theorems \cite{[PeWe]},
heat kernel estimates \cite{[DW]}, isoperimetric inequalities \cite{[Gal],[PeSp]}, Colding's
volume convergence and Cheeger-Colding splitting theorems \cite{PeWe00}, local Sobolev
constant estimates \cite{[DWZ]}, etc. are all possibly extended to the case where the
weighted integral of $\mathrm{Ric}^m_f$ bounds.

\bibliographystyle{amsplain}

\end{document}